\documentclass[12pt,reqno]{amsart}
\usepackage[centering,bottom=3.5cm,top=3.4cm]{geometry}
\usepackage[foot]{amsaddr}

\usepackage{amsmath,amssymb,amsthm,algorithm,algorithmicx,algpseudocode}
\usepackage[inline]{enumitem}
\usepackage{thmtools}
\usepackage{hyperref,xcolor,graphicx}
\usepackage[capitalize,nameinlink,sort]{cleveref}
\usepackage{subcaption}
\usepackage{hyperref}
\usepackage{cleveref}
\usepackage{float}
\allowdisplaybreaks

\hypersetup{hidelinks}

\declaretheorem[numberwithin=section]{theorem}
\declaretheorem[sibling=theorem]{lemma,corollary,proposition}
\declaretheorem[sibling=theorem,style=definition]{example,definition,remark}

\usepackage{pgf,tikz}
\usetikzlibrary{arrows, automata, calc, chains, decorations, shapes, positioning}
\usetikzlibrary{arrows,decorations.pathmorphing}

\newcommand{\B}{{\mathbf{B}}}
\DeclareMathOperator{\val}{val}
\newcommand{\lex}{{\mathrm{lex}}}

\newcommand\nnfootnote[1]{%
  \begin{NoHyper}
  \renewcommand\thefootnote{}\footnote{#1}%
  \addtocounter{footnote}{-1}%
  \end{NoHyper}
}

\begin{document}

\title[Computing Expansions Cantor Real Bases via a Transducer]{Computing Expansions in Infinitely Many Cantor Real Bases via a Single Transducer}

\author[\'{E}. Charlier]{\'{E}milie Charlier}
\author[P. Popoli]{Pierre Popoli}
\author[M. Rigo]{Michel Rigo}

\address{Department of Mathematics, ULiège, Liège, Belgium}
\email{echarlier@uliege.be, pierre.popoli@uliege.be, m.rigo@uliege.be}

\begin{abstract}
Representing real numbers using convenient numeration systems (integer bases, $\beta$-numeration, Cantor bases, etc.) has been a longstanding mathematical challenge. This paper focuses on Cantor real bases and, specifically, on automatic Cantor real bases and the properties of expansions of real numbers in this setting. 

We develop a new approach where a single transducer associated with a fixed real number~$r$, computes the $\mathbf{B}$-expansion of~$r$ but for an infinite family of Cantor real bases~$\mathbf{B}$ given as input. This point of view contrasts with traditional computational models for which the numeration system is fixed. Under some assumptions on the finitely many Pisot numbers occurring in the Cantor real base, we show that only a finite part of the transducer is visited. 

We obtain fundamental results on the structure of this transducer and on decidability problems about these expansions, proving that for certain classes of Cantor real bases, key combinatorial properties such as greediness of the expansion or periodicity can be decided algorithmically.
\end{abstract}

\maketitle

\nnfootnote{\textbf{Keywords}: Cantor real base ; Pisot number ; Transducer ; Automatic sequence ; Decision problem.}

\nnfootnote{\textbf{MSC 2020 Subject Classification}: 11A63, 11K16; 11B85, 68Q45, 68R15}

\setcounter{tocdepth}{1}
\tableofcontents

\section{Introduction}

There are multiple ways to represent real numbers by bounded sequences of non-negative integers. We first recall the two classical definitions of $\beta$-numeration and Cantor base. Then, we show how Cantor real bases unify these two concepts. Finally, we detail our contributions to Cantor real bases.

\subsection{\texorpdfstring{$\beta$-expansions and Cantor bases}{beta-expansions and Cantor bases}}

Let $\beta>1$ be a real number. On the one hand, a {\em $\beta$-representation} of a real number $r\in [0,1]$ is a sequence $(a_n)_{n\in\mathbb{N}}$ of non-negative integers such that 
\begin{equation}
\label{eq:beta-exp}
    r=\sum_{n=0}^{+\infty} \frac{a_n}{\beta^{n+1}}.
\end{equation} 
These representations were first studied by Rényi~\cite{Renyi1957} and Parry~\cite{Parry1960} with sequences taking values in $\{0,\ldots,\lceil \beta \rceil -1\}$. Whereas a real number usually has an uncountable number of $\beta$-representations~\cite{Sidorov2003}, one of them is distinguished. It is computed by the greedy algorithm thanks to the so-called {\em $\beta$-transformation} 
\begin{equation}
\label{eq:Tbeta}
    \mathsf{T}_{\beta}\colon [0,1]\to [0,1),\ r\mapsto \beta r-\lfloor \beta r\rfloor.
\end{equation}
Note that we also defined this map on $1$, which is unusual; this choice was made intentionally because it will facilitate the presentation of the transducer $\mathcal{T}$ in~\cref{sec:main_transducer}.
By letting $a_n=\lfloor \beta \mathsf{T}_{\beta}^n(r) \rfloor$ for all $n\ge 0$, the obtained representation is called the  \emph{(greedy) $\beta$-expansion} of $r$ and is denoted $d_{\beta}(r)$. It is maximal in lexicographic order among all $\beta$-representations of $r$ and is characterized by the inequalities
\begin{equation}
\label{eq:greedy-inequalities}
    \sum_{n=\ell}^{+\infty} \frac{a_n}{\beta^{n+1}}<\frac{1}{\beta^{\ell}}\qquad \forall\ell\ge 1.
\end{equation}
For a general reference on $\beta$-numeration systems, see~\cite[Chap.~6]{Lothaire}. For some related computational questions, see, for instance, \cite{MR3230872,MR3404455,MR2299792}.

On the other hand, Cantor~\cite{cantor} introduced another type of expansion of real numbers based on a sequence $(q_n)_{n\in \mathbb{N}}$ of integers greater than or equal to $2$. Every $r\in[0,1)$ has a unique expansion of the form 
\[ 
    r=\sum_{n=0}^{+\infty} \frac{a_n}{q_0\cdots q_n}
\]
where $0\le a_n<q_n$ and $a_n\neq q_n-1$ infinitely often. 

\subsection{Cantor real base}
The framework of Cantor real bases unifies $\beta$-expansions with a single real base $\beta$ and Cantor bases built on a sequence of integers. 

\begin{definition}
    A \emph{Cantor real base} is a sequence $\B=(\beta_n)_{n\in\mathbb{N}}$ of real numbers greater than~$1$  such that $\prod_{n=0}^{+\infty}\beta_n=+\infty$. If $\B$ is a constant sequence, we recover the numeration systems in a real base $\beta$. If $\B$ is a sequence of integers, we get back to classical Cantor bases.
\end{definition}
We let $\val_{\B}\colon \mathbb{N}^\mathbb{N}\to\mathbb{R}$ denote the (partial) {\em $\B$-valuation} map, i.e., if $(a_n)_{n\in\mathbb{N}}$ is a sequence of natural numbers, we set 
\[
    \val_{\B}(a_0a_1\cdots):=\sum_{n=0}^{+\infty}\frac{a_n}{\beta_0\cdots \beta_n}
\]
provided that the series converges. A sequence $(a_n)_{n\in\mathbb{N}}$ over $\mathbb{N}$ is said to be a {\em $\B$-representation} of $r\in\mathbb{R}$ whenever $\val_{\B}(a_0a_1\cdots)=r$. 

\begin{definition}[Greedy algorithm]
Let $r\in [0,1]$. Let $a_0,\ldots,a_{N-1}$ be the first~$N\ge 0$ digits of the greedy $\B$-representation of $r$. The next digit $a_N$ of the greedy $\B$-representation of $r$ is the largest integer such that 
\[ 
    \sum_{n=0}^N	\frac{a_n}{\beta_0\cdots \beta_n}\le r.
\] 
For all $n\in \mathbb{N}$, the digit $a_n$ belongs to the alphabet $\{0,\ldots,\lceil{\beta_n}\rceil-1\}$. The corresponding $\B$-representation is called the (greedy) {\em $\B$-expansion} of~$r$, and is denoted by $d_{\B}(r)$.     

The greedy algorithm can be used in order to obtain a distinguished $\B$-representation of $r\in[0,1]$: for all $n\ge 0$, we let $a_n=\lfloor \beta_n \mathsf{T}_{n-1}\cdots \mathsf{T}_0(r)\rfloor$. The corresponding $\B$-representation is called the (greedy) {\em $\B$-expansion} of~$r$, and is denoted by $d_{\B}(r)$.  It is maximal in the lexicographic order among all the $\B$-representations of $r$ and it is characterized by the inequalities
\begin{equation}
    \label{eq:greedy-inequalities-Cantor}
    \sum_{n= \ell}^{+\infty} \frac{a_n}{\beta_0\cdots \beta_n}<\frac{1}{\beta_0\cdots \beta_{\ell-1}} \qquad \forall\ell\ge 1.
\end{equation}
\end{definition}

\begin{definition}[Alternate bases]
   If the Cantor real base $\B=(\beta_0,\ldots,\beta_{p-1})^\omega$ is a periodic sequence, we say that $\B$ is an {\em alternate base}. These systems have been extensively studied in~\cite{CC2021,CCK,CCD,CCMP}. In this paper, we will also encounter a slightly more general case. If $\B$ is ultimately periodic, we say that $\B$ is an {\em ultimately alternate base}.
\end{definition}


\subsection{Our contribution}
Our initial motivation is to study Cantor real bases that are more general than the alternate ones. From the perspective of combinatorics on words, it seems natural to first consider sequences with low factor complexity that have a rich combinatorial structure. Thus, the initial approach of this work was to consider a Cantor real base defined by an automatic sequence. For details on automatic sequences, we refer to~\cite{AS03}.

Let $E$ be a finite set, called {\em alphabet}, of real numbers greater than $1$. We consider Cantor real bases $\B$ belonging to $E^{\mathbb{N}}$. In the initial spirit of Cantor, who considered only products of integers, we first examine automatic Cantor bases~$\B$ over an alphabet included in $\mathbb{N}_{\ge 2}$. In that setting, with~\cref{cor:betahc}, our main result is the following. Consider a family of Cantor bases $\B$ obtained as images under a monoid morphism~$\psi$ over an alphabet included in $\mathbb{N}_{\ge 2}$, not necessarily uniform but with the property that there exists a constant~$\delta$ such that, for all letters $a$, the product of all letters occurring in $\psi(a)$ is equal to $\delta$. In particular, this setting covers the Cantor bases obtained as images of Parikh-constant morphisms. If $\B=\psi(\mathbf{a})$ where $\mathbf{a}$ is any infinite word, then $\B$-expansion of any real number $r$ in $[0,1)$ is obtained from the $\delta$-expansion of~$r$ by applying some well-defined morphisms in the order prescribed by the preimage infinite word $\mathbf{a}$. To simplify the reading of this paper and carry on intuition,  in~\cref{sec:Cantorbaseintegers}, we first address the case of a Thue--Morse Cantor base over $\{2,3\}$.  
Then we prove that, for any fixed rational number $r\in [0,1)$, one can associate a single finite transducer that outputs the $\B$-representation of~$r$ for infinitely many Cantor bases given as input.

\cref{sec:main_transducer} contains, with~\cref{thm:main_transducer}, what we consider to be the main contribution of this work. Let $\delta$ be an algebraic integer and let $E$ be a finite alphabet of Pisot numbers sharing the same algebraic degree and belonging to $\mathbb{Z}[\delta]$. For all $r\in\mathbb{Q}(\delta)\cap [0,1]$, there exists a finite letter-to-letter transducer\footnote{A {\tt Mathematica} implementation is available on \url{https://hdl.handle.net/2268/333750}.} which produces $d_\B(r)$ when fed with any Cantor real base $\B\in E^\mathbb{N}$. This result changes the commonly used perspective. Typically, computational machines are designed for a fixed numeration system and operate on representations of infinitely many real numbers. In contrast, our original approach fixes a single number and considers its representations across infinitely many bases.

In the next section, under the same assumption as above, we show with~\cref{pro:2loops} that the Cantor real bases $\B\in E^\mathbb{N}$ such that the quasi-greedy expansion $d_{\B}^*(r)$ is ultimately periodic are exactly the ultimately alternate bases if and only if the transducer $\mathcal{T}^*_{E,r}$ has a special form. Roughly speaking, it does not contain two closed walks with different inputs and the same output. This $2$-walk property is decidable and we discuss some examples. For instance, in~\cref{ex:smallPisot} we show that $d_\B^*(1)$ is purely periodic for uncountably many aperiodic Cantor bases built from the smallest Pisot number. 

Since the structure of the transducer determines the kinds of $\B$-expansions that we may have, it is natural to look at its connectedness. In the short~\cref{sec:connectedness}, we give examples showing that all situations may appear: the Cantor bases being built with simple and/or non-simple Parry numbers.

Finally, \cref{sec:decidability} explores questions related to decidability. Returning to our initial question when the Cantor base is automatic, we derive from our main theorem on the finiteness of the transducers that if $\B$ is an automatic sequence, then many combinatorial properties of $d_{\B}^*(1)$, such as periodicity or power avoidance, become decidable. Actually, we may ask many questions considered by Shallit and his co-authors \cite{ShallitBook}. This leads us to the following question: given an infinite word --- typically an automatic sequence, since the key issue is how the word is specified using a finite amount of information --- can we determine whether it is $\B$-admissible, i.e., a valid $\B$-expansion? We give several cases where this question admits a positive resolution (for instance, including an example with a transcendental base). This section relies on B\"uchi's theorem on the decidable first-order theory of $\langle\mathbb{N},+,V_b\rangle$ \cite{Buchi1960,ShallitBook}.


\section{Cantor base over an alphabet of integers} \label{sec:Cantorbaseintegers}

The arguments developed in this section are rather straightforward. They rely on elementary arithmetic. However, these results highlight the flavor of this paper: obtaining the $\B$-expansion of a real number by transforming an infinite word into another. In this paper, we use both the notation $\beta_0\beta_1\cdots$ and $(\beta_0,\beta_1,\ldots)$ to designate a Cantor real base $\B=(\beta_n)_{n\in\mathbb{N}}$. The second will be mostly used whenever we need avoiding ambiguity between product of numbers and concatenation of letters.

\subsection{Working out a classical example} 
Consider the finite alphabet of numbers $E=\{2,3\}$ and let ${\mathbf{T}}$ be the celebrated Thue--Morse sequence
\begin{equation*}
    {\mathbf{T}}=23323223322323322\cdots,    
\end{equation*}
over $E$ starting with $2$, which is an infinite fixed point of the morphism 
\begin{equation}
\label{eq:TM-morphism}
    \tau\colon E\to E^*,\ 2\mapsto 23, 3\mapsto 32.
\end{equation} 
The general question is the following one: {\em What can be said about the structure of ${\mathbf{T}}$-expansions?} We give a first answer here and a second one with \cref{cor:automatic}. 

Every real number $r\in[0,1)$ has a (usual) $6$-expansion $(c_n)_{n\in\mathbb{N}}$ in the sense of \eqref{eq:beta-exp} and~\eqref{eq:greedy-inequalities}. 

\begin{definition}
\label{def:h2h3}
Every $c\in\{0,\ldots,5\}$ can be uniquely decomposed as
\[c=a\cdot 3+b\quad \text{ with }\quad a\in\{0,1\},\ b\in\{0,1,2\}.\]
We define the morphism $h_2:c\mapsto ab$, i.e.,
\[h_2\colon 0\mapsto 00,\ 1\mapsto 01,\ 2\mapsto 02,\ 3\mapsto 10,\ 4\mapsto 11,\ 5\mapsto 12.\]
Similarly, every $c\in\{0,\ldots,5\}$ can be uniquely decomposed as
\[c=a \cdot 2+b\quad \text{ with }\quad a\in\{0,1,2\},\ b\in\{0,1\}.\]
We define the morphism $h_3:c\mapsto ab$, i.e.,
\[h_3\colon 0\mapsto 00,\ 1\mapsto 01,\ 2\mapsto 10,\ 3\mapsto 11,\ 4\mapsto 20,\ 5\mapsto 21.\]
\end{definition}

\begin{theorem} \label{thm:TM23}
For all $r\in[0,1)$, the ${\mathbf{T}}$-expansion of $r$ is the sequence 
\[
    h_2(c_0)h_3(c_1)h_3(c_2)h_2(c_3)\cdots 
\]
obtained from $d_6(r)=c_0c_1c_2c_3\cdots$ by application of the morphisms $h_2$ and $h_3$ in the ordering given by the Thue--Morse word~${\mathbf{T}}$. 
\end{theorem}

\begin{proof}
Write ${\mathbf{T}}=(\beta_n)_{n\in\mathbb{N}}$, and let 
\[
    \mathbf{y}=y_0y_1y_2\cdots=h_2(c_0)h_3(c_1)h_3(c_2)h_2(c_3)\cdots h_{\beta_n}(c_n)\cdots.
\]
Otherwise stated, we set $h_{\beta_n}(c_n)=y_{2n}y_{2n+1}$ for all $n\ge 0$. We show that $\val_{\mathbf{T}}(\mathbf{y})=r$ and 
\[
    \forall\ell\ge 1,\quad  \sum_{n=\ell}^{+\infty} \frac{y_n}{\prod_{i=0}^n\beta_i}
    <\frac{1}{\prod_{i=0}^{\ell-1} \beta_i}
\]
proving that $\mathbf{y}$ is indeed get the ${\mathbf{T}}$-expansion of $r$. Let us first prove that $\mathbf{y}$ is a $\mathbf{T}$-representation of $r$. We have
\begin{eqnarray*}
    \val_{\mathbf{T}}(\mathbf{y}) 
    &=& \sum_{n= 0}^{+\infty}\left( \frac{y_{2n}}{\prod_{i=0}^{2n}\beta_i}
    +\frac{y_{2n+1}}{\prod_{i=0}^{2n+1}\beta_i}\right)
    =\sum_{n=0}^{+\infty}\frac{1}{6^{n+1}} (y_{2n}\beta_{2n+1}+y_{2n+1}).
\end{eqnarray*}

It suffices to show that $c_n=y_{2n}\beta_{2n+1}+y_{2n+1}$.
We have two cases to consider regarding the value of $\beta_n$. If $\beta_n=2$ then, on the one hand, $c_n= y_{2n}\cdot 3+y_{2n+1}$ and on the other hand, $\tau(\beta_n)=\beta_{2n}\beta_{2n+1}=23$, hence $\beta_{2n+1}=3$. Similarly, if $\beta_n=3$ then $c_n= y_{2n}\cdot 2+y_{2n+1}$ and $\beta_{2n+1}=2$ since $\tau(\beta_n)=\beta_{2n}\beta_{2n+1}=32$.

We now turn to the greediness of the $\mathbf{T}$-representation. Let us first assume that the tail is starting with an even index~$2\ell$. Then
\[
    \sum_{n= 2\ell}^{+\infty} \frac{y_n}{\prod_{i=0}^n\beta_i}
    =\sum_{n= \ell}^{+\infty} \left( \frac{y_{2n}}{\prod_{i=0}^{2n}\beta_i}
    +\frac{y_{2n+1}}{\prod_{i=0}^{2n+1}\beta_i}\right)
    =\sum_{n= \ell}^{+\infty} \frac{c_n}{6^{n+1}}
    <\frac{1}{6^{\ell}}
    =\frac{1}{\prod_{i=0}^{2\ell-1} \beta_i},
\]
where we used the fact that $(c_n)_{n\in\mathbb{N}}$ is the greedy $\delta$-expansion of $r$. Assume now that the tail is starting with an odd index~$2\ell+1$. We get
\[
    \sum_{n=2\ell+1}^{+\infty} \frac{y_n}{\prod_{i=0}^n\beta_i}
    =\frac{y_{2\ell+1}}{\prod_{i=0}^{2\ell+1}\beta_i} 
    + \sum_{n=\ell+1}^{+\infty} \left( \frac{y_{2n}}{\prod_{i=0}^{2n}\beta_i}
    +\frac{y_{2n+1}}{\prod_{i=0}^{2n+1}\beta_i}\right)
    =\frac{y_{2\ell+1}}{6^\ell} +\sum_{n\ge \ell+1} \frac{c_n}{6^{n+1}}.
\]
Again, we have two cases. If $\beta_\ell=2$ then 
$y_{2\ell+1}\le 2$ by definition of the morphism $h_2$. Using that $(c_n)_{n\in\mathbb{N}}$ is a $\delta$-expansion and that $\beta_{2\ell}=2$, we see that the above quantity is less than
\[
    \frac{2}{6^{\ell+1}}+\frac{1}{6^{\ell+1}}
    =\frac{3}{2\cdot 3\cdot 6^\ell}=\frac{1}{\beta_{2\ell}\cdot 6^\ell}
    =\frac{1}{\prod_{i=0}^{2\ell}\beta_i}.
\]
If $\beta_\ell=3$ then $y_{2\ell+1}\le 1$ by definition of the morphism $h_3$. Similarly, the above quantity is less than
\[
    \frac{1}{6^{\ell+1}}+\frac{1}{6^{\ell+1}}
    =\frac{2}{2\cdot 3 \cdot 6^\ell}
    =\frac{1}{\beta_{2\ell}\cdot 6^\ell}
    =\frac{1}{\prod_{i=0}^{2\ell}\beta_i}.
\]
\end{proof}

\subsection{Generalization to Cantor bases built from blocks of constant product}

The Thue--Morse case can be readily generalized to the following family of Cantor bases. 

\begin{definition}
\label{def:blocks-delta}
Let $\delta\ge 2$ be an integer. For each $k\in\mathbb{N}$, let $\ell_k\ge 1$ be an integer and $\gamma_{k,0},\ldots,\gamma_{k,\ell_k-1}\ge 2$ be integers whose product equal to $\delta$: 
\begin{equation}
    \label{eq:product-delta}
    \prod_{j=0}^{\ell_k-1}\gamma_{k,j}=\delta.
\end{equation}
Let $r$ be a real number in $[0,1]$ with $d_\delta(r)=c_0c_1c_2\cdots$. 
For each $k\in\mathbb{N}$, every $c_k$ belongs to $\{0,\ldots,\delta-1\}$ and can be uniquely written as
  \begin{equation}
\label{eq:coeff-x_k}    
    c_k=\sum_{j=0}^{\ell_k-1}a_{k,j}\cdot \gamma_{k,j+1}\cdots \gamma_{k,\ell_k-1}
  \end{equation}
  where for each $j$, the coefficient $a_{k,j}$ is a non-negative integer less than $\gamma_{k,j}$. 
\end{definition}

\begin{theorem}
\label{thm:blocks-delta}
With the notation of~\cref{def:blocks-delta}, consider the Cantor base 
\[
    \B=\beta_0\beta_1\beta_2\cdots=\prod_{k=0}^{+\infty} (\gamma_{k,0}\cdots \gamma_{k,\ell_k-1})
\]
(where the product of the right-hand side is the concatenation product). Then the $\B$-expansion of $r$ is the sequence 
\[
    d_{\B}(r)=\prod_{k=0}^{+\infty} (a_{k,0}\cdots a_{k,\ell_k-1}).
\]
\end{theorem}

\begin{proof}
The proof follows the same lines as the proof of~\cref{thm:TM23}. We have $d_\delta(r)=c_0c_1c_2\cdots$, and we let 
\[
    \mathbf{y}=y_0y_1y_2\cdots
   =\prod_{k=0}^{+\infty} (a_{k,0}\cdots a_{k,\ell_k-1}).
\]
First, we show that $\mathbf{y}$ is a $\B$-representation of $r$. By decomposing into blocks of length $\ell_k$ and by using~\eqref{eq:product-delta} and~\eqref{eq:coeff-x_k}, we compute
\begin{align*}
    \val_{\B}(\mathbf{y})
    &=\sum_{k=0}^{+\infty} \sum_{j=0}^{\ell_k-1}\frac{y_{\ell_0+\cdots+\ell_{k-1}+j}}{\prod_{i=0}^{\ell_0+\cdots+\ell_{k-1}+j}\beta_i}
    =\sum_{k=0}^{+\infty} \frac{1}{\delta^k}\sum_{j=0}^{\ell_k-1}\frac{a_{k,j}}{\prod_{i=0}^j\gamma_{k,i}},\\
    &=\sum_{k=0}^{+\infty} \frac{1}{\delta^{k+1}}\sum_{j=0}^{\ell_k-1}(a_{k,j}\cdot \gamma_{k,j+1} \cdots \gamma_{k,\ell_k-1})
    =\sum_{k=0}^{+\infty}\frac{c_k}{\delta^{k+1}}=r.
\end{align*}

Now let us check the greediness of the expansion, i.e., the inequalities given in~\eqref{eq:greedy-inequalities-Cantor}. Let $\ell\ge 1$. We have to prove that 
\[
    \sum_{n= \ell}^{+\infty} \frac{y_{n}}{\prod_{i=0}^n\beta_i}
    <\frac{1}{\prod_{i=0}^{\ell -1}\beta_i}.
\]
Again, we need to decompose our computation into blocks of length $\ell_k$. There exists a unique $K\in\mathbb{N}$ such that 
\[
    \sum_{k=0}^{K-1}\ell_k\le \ell< \sum_{k=0}^K\ell_k.
\]
Let 
\[
    \ell'=\ell-\sum_{k=0}^{K-1}\ell_k.
\]
Thus, $\ell'\in\{0,\ldots,\ell_K-1\}$. By the same reasoning as above, we get
\begin{equation*}
  \label{eq:part1}
  \sum_{n=\ell}^{+\infty} \frac{y_n}{\prod_{i=0}^n\beta_i}
  =\frac{1}{\delta^{K+1}} \sum_{j=\ell'}^{\ell_K-1} a_{K,j}\cdot \gamma_{K,j+1}\cdots\gamma_{K,\ell_{K}-1}
  +\sum_{k=K+1}^{+\infty}\frac{c_k}{\delta^{k+1}}.
\end{equation*}
Since $c_0c_1c_2\cdots$ is the $\delta$-expansion of $r$, we get from~\eqref{eq:greedy-inequalities} that 
\[
    \sum_{k=K+1}^{+\infty}\frac{c_k}{\delta^{k+1}}<\frac{1}{\delta^{K+1}}.
\]
The greediness of~\eqref{eq:coeff-x_k} implies that
\[
    \sum_{j=\ell'}^{\ell_K-1} a_{K,j}\cdot \gamma_{K,j+1}\cdots\gamma_{K,\ell_K-1}
    \le \gamma_{K,\ell'}\cdots\gamma_{K,\ell_K-1}-1.
\]  
Putting together the latter two inequalities, we get
\[
     \sum_{n= \ell}^{+\infty} \frac{y_{n}}{\prod_{i=0}^n\beta_i}
     <\frac{\gamma_{K,\ell'}\cdots\gamma_{K,\ell_{K}-1}}{\delta^{K+1}}
     =\frac{1}{\delta^K\gamma_{K,0}\cdots\gamma_{K,\ell'-1}}
     =\frac{1}{\prod_{i=0}^{\ell-1}\beta_i}.
\]
\end{proof}

As a particular case, this result applies to the following situation. 

\begin{definition}
\label{def:betahc}
Let $\delta\geq 2$ be an integer, and let $A,E\subset \mathbb{N}_{\ge 2}$ be finite alphabets. Let $\psi\colon A^* \to E^*$ be a morphism with the property that for each letter $a\in A$, the product of the letters occurring in the image $\psi(a)$ is equal to $\delta$: if $\psi(a)=b_0\cdots b_{\ell-1}$ then 
\[
    \prod_{j=0}^{\ell-1}b_j=\delta.
\]
For each letter $a\in A$, we define a morphism $h_a$ as follows. Let $a\in A$ and $\psi(a)=b_0\cdots b_{\ell-1}$. Every $c\in\{0,\ldots,\delta-1\}$ can be uniquely written as
  \[
    c=\sum_{j=0}^{\ell-1}c_j\cdot b_{j+1}\cdots b_{\ell-1}
  \]
  where $c_0,\ldots,c_{\ell-1}$ are non-negative integers such that $c_j\le b_j-1$ for all $j$. This defines a uniform morphism 
  \begin{equation*}
    \label{eq:hc}
    h_a\colon\{0,\ldots,\delta-1\}^*\to \{0,\ldots,(\max E)-1\}^*,\ c\mapsto c_0\cdots c_{\ell-1}
   \end{equation*}
   associated with each $a\in A$.
\end{definition}

As an example, take $\delta=72$, $A=\{2,3,4\}$ and $E=\{2,3,4,6\}$. Consider the morphism $\psi\colon 2\mapsto 634,\ 3\mapsto 3243,\ 4\mapsto 4332$. For each $a\in A$, we have $h_a\colon \{0,\ldots,71\}^*\to \{0,\ldots,5\}^*$. Since 
\begin{align*}
    &61=5 \cdot 12+ 0\cdot 4 +1, \\
    &61=2\cdot 24 +1\cdot 12 +0\cdot 3 +1, \\
    &61=3\cdot 18 +1\cdot 6 +0\cdot 2+1,
\end{align*}
we have $h_2(61)=501$, $h_3(61)=2101$ and $h_4(61)=3101$. 

Note that any Parikh-constant morphism obviously satisfies this property, but for our considerations, we do not even need $\psi$ to be a uniform morphism. 

\begin{corollary}
    \label{cor:betahc}
With the notation of~\cref{def:betahc}, consider a Cantor base $\B=\psi(a_0a_1a_2\cdots)$ where $a_0a_1a_2\cdots$ is any sequence over $A$. For any real number $r\in [0,1]$, the $\B$-expansion of $r$ is the sequence 
$h_{a_0}(c_0)h_{a_1}(c_1)h_{a_2}(c_2)\cdots$ obtained from $d_\delta(r)=c_0c_1c_2\cdots$ by application of the morphisms $h_a$, $a\in A$, in the ordering given by the infinite word $a_0a_1a_2\cdots$.
\end{corollary}

 Observe that in~\cref{thm:TM23}, the morphisms $h_2$ and $h_3$ are applied in the ordering given by the Cantor base $23323223\cdots$ itself. This is due to the fact that the Thue--Morse word is a fixed point of the morphism $\tau$, which is a stronger property than the hypotheses of~\cref{cor:betahc}.

Even though the example below may seem unnecessary, it already highlights the philosophy developed in the next section: having a transducer that takes a Cantor base as input and outputs the corresponding $\B$-expansion of a fixed real number.

\begin{example}[A first transducer]
\label{exa:first_transducer}
Consider the rational number $r=932/3885$ whose $6$-expansion is the ultimately periodic word $d_6(r)=1(2345)^\omega$. This finite information is stored in a ``frying pan'' transducer depicted in~\cref{fig:transducer-fig} where $a\in\{2,3\}$ and $h_a$ are the two morphisms from~\cref{def:h2h3}. 
\begin{figure}[ht]
   \begin{center}
    \begin{tikzpicture}[->,thick,main node/.style={circle,draw,minimum size=12pt,inner sep=2pt}]
  
        \node[main node] (A) at (90:2) {3};
        \node[main node] (B) at (0:2) {4};
        \node[main node] (C) at (-90:2) {5};
        \node[main node] (D) at (180:2) {2};
        \node[main node] (E) at (180:5) {1};

        \draw[->] (A) -- (B) node[pos=.4, right, yshift=2pt] {$a|h_a(3)$};
        \draw[->] (B) -- (C) node[pos=.6, right, yshift=-3pt] {$a|h_a(4)$};
        \draw[->] (C) -- (D) node[pos=.4, left, xshift=-3pt] {$a|h_a(5)$};
        \draw[->] (D) -- (A) node[pos=.6, left, yshift=2pt] {$a|h_a(2)$};
        
        \draw[->] (E) -- (D) node[midway, above] {$a|h_a(1)$};
    \end{tikzpicture}
\end{center}
 \caption{A finite transducer for $d_{\mathbf{T}}(932/3885)$, $a\in\{2,3\}$.}
    \label{fig:transducer-fig}
\end{figure}
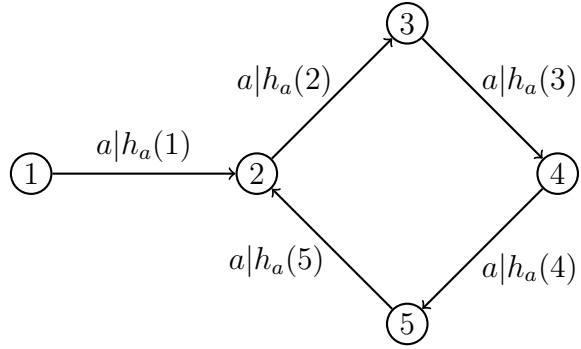
A label $a|w$ means that $a$ is read as input and $w$ is the corresponding output word. This means that edges have two input/output labels, one for each letter in $\{2,3\}$. \cref{thm:TM23} can be applied as follows. Starting from the initial state $1$, read the Thue--Morse word as input and follow the transitions of the transducer. One trivially obtains 
\[
    d_{\mathbf{T}}(r)
    =h_2(1)h_3(2)h_3(3)h_2(4)h_3(5)h_2(2)h_2(3)h_3(4)\cdots=0110111121021020\cdots
\] 
as output. More interestingly, due to~\cref{cor:betahc}, the same transducer will produce $d_\B(932/3885)$ whenever $\B$ is a Cantor base of the form $\tau(a_0a_1a_2\cdots)$ where $a_i\in\{2,3\}$ and $\tau$ is the morphism given in~\eqref{eq:TM-morphism}. Indeed, with the input sequence $a_0a_1a_2\cdots$, the output is $h_{a_0}(1)h_{a_1}(2)h_{a_2}(3)h_{a_3}(4)h_{a_4}(5)h_{a_5}(2)\cdots$. We thus have a single machine associated with a given number (with a periodic expansion) that allows us to produce the $\B$-representation of this number for an infinite number of Cantor bases $\B$ of the form $\tau(a_0a_1a_2\cdots)$ for all $a_i\in\{2,3\}$. 

Note that the transducer in \cref{fig:transducer-fig} is not letter-to-letter: when reading a symbol $a$, the output has length~$2$. In this paper, all the other transducers will be letter-to-letter. Since $\tau,h_2,h_3$ are $2$-uniform, we can transform this transducer into a letter-to-letter transducer by replacing every edge of the form
\[
    q \stackrel{2|h_2(k)}{\xrightarrow{\hspace{2 cm}}} s 
    \quad\text{ with }\quad 
    q \stackrel{2|[h_2(k)]_0}{\xrightarrow{\hspace{2 cm}}} p \stackrel{3|[h_2(k)]_1} {\xrightarrow{\hspace{2 cm}}} s,
\]
and 
\[
    q \stackrel{3|h_3(k)}{\xrightarrow{\hspace{2 cm}}} s 
    \quad\text{ with }\quad 
    q \stackrel{3|[h_3(k)]_0}{\xrightarrow{\hspace{2 cm}}} p' \stackrel{2|[h_3(k)]_1} {\xrightarrow{\hspace{2 cm}}} s,
\]
where $p,p'$ are two new intermediate states and $[w]_i$ denotes the $i$-th symbol occurring in $w$.

With this new transducer, the input to consider is $\tau(a_0a_1a_2\cdots)$ and not the preimage as above. Indeed, we have replace an edge of label $2$ (resp.\ $3$) with two edges of labels $2$ and $3$ (resp.\ $3$ and $2$).
\end{example}

Since rational numbers are characterized by an ultimately periodic expansion in an integer base, the following observation is obvious. 
\begin{corollary}
    With the assumptions of~\cref{cor:betahc}. An infinite word is the $\B$-expansion of a rational number in $[0,1]$ if and only if it is the image by application of the morphisms $h_a$, $a\in A$, in the ordering given by $a_0a_1a_2\cdots$ where $\B=\psi(a_0a_1a_2\cdots)$, of an ultimately periodic word over $\{0,\ldots,\delta-1\}$.
\end{corollary}

More interestingly, as in~\cref{exa:first_transducer}, we can associate with any rational number~$r\in [0,1)$ a unique transducer that allows us to produce the $\B$-representation of~$r$ for infinitely many Cantor bases. With the notation of~\cref{cor:betahc}, the input sequence can be any infinite word $a_0a_1a_2\cdots$ corresponding to the Cantor base $\B=\psi(a_0a_1a_2\cdots)$. We can then transform this transducer so that it becomes letter-to-letter.


\section{One transducer for uncountably many Cantor real bases}\label{sec:main_transducer}

For a general reference on transducers, see~\cite{Saka}. From now on  transducers that we are considering are letter-to-letter. We will build infinite transducers that mimic the greedy and quasigreedy algorithms that compute particular $\B$-representations. Then we show that, under some assumptions, only a finite part of these machines is visited. This finiteness property implies that if $\B$ is an automatic Cantor base, then the $\B$-expansion of $1$ is also automatic. In the second part of this section, we build various examples.

\subsection{Two infinite transducers}
We make use of the notation~$\mathsf{T}_\beta$ from \eqref{eq:Tbeta}.

\begin{definition}[Greedy transducer]
Let $E\subset\mathbb{R}_{>1}$ be a finite alphabet. 
We define an infinite transducer denoted by~$\mathcal{T}_E$. 
\begin{itemize}
\item The set of states is $[0,1]$; 
\item The input alphabet is $E$;
\item The output alphabet is $\mathbb{N}$;
\item For all states $r\in [0,1]$ and all letters $\beta \in E$, there is a transition
  \[
    r\stackrel{\beta \,|\, \lfloor \beta\, r\rfloor}{\xrightarrow{\hspace{3 cm}}} \mathsf{T}_{\beta}(r)
    \]
 \end{itemize}
\end{definition}

It is often useful to also consider the quasi-greedy expansions. When the $\B$-expansion of $1$ is finite, Charlier and Cisternino introduced the quasi-greedy $\B$-expansion of $1$, see~\cite{CC2021}. For our purpose, we extend this notion to any element of the unit interval $[0,1]$. The \emph{quasi-greedy} $\B$-expansion of $r\in (0,1]$ is defined by 
\[
    d_\B^*(r)=\lim_{z\to r^-} d_\B(z),
\]
where the limit on the r.h.s.\ is for the product topology on the set of infinite words. It will be convenient to also set $d_\B^*(0)=0^\omega$.

We let $\sigma$ denote the {\em shift operator} on sequences, i.e., for a sequence $(a_n)_{n\in\mathbb{N}}$, we have $\sigma((a_n)_{n\in\mathbb{N}})=(a_{n+1})_{n\in\mathbb{N}}$. In particular, if $\B$ is a Cantor base, then $\sigma(\B)$ is again a Cantor base. The quasi-greedy $\B$-expansion of $r\in[0,1]$ can be recursively computed as
 \[
    d_{\B}^*(r):=\begin{cases}
        d_{\B}(r), & \text{if $d_{\B}(r)$ is infinite or equal to $0^\omega$}; \\ 
        a_0 \cdots a_{\ell-2}(a_{\ell-1}-1)d_{\sigma^{\ell}(\B)}^*(1), &\text{if $d_{\B}(r)= a_0 \cdots a_{\ell-1}$ with $a_{\ell-1}>0$}.
    \end{cases}  
\] 

\begin{definition}[Quasi-greedy transducer]
As a quasi-greedy variant, instead of $\mathsf{T}_\beta$ we consider the transformation 
\[
    \mathsf{T}_\beta^*\colon [0,1]\to (0,1],\ 
    \begin{cases}
        r\mapsto \beta r -\lceil \beta r-1\rceil, & \text{if } r\ne 0;\\
        0,   & \text{if } r=0,
    \end{cases}
\]
and the corresponding infinite transducer~$\mathcal{T}^*_E$ whose transitions are
\[r\stackrel{\beta \, |\, \lceil \beta\, r-1\rceil}{\xrightarrow{\hspace{3 cm}}} \mathsf{T}_{\beta}^*(r).\]  
\end{definition}


It is important to note that both transducers $\mathcal{T}_E$ and $\mathcal{T}^*_E$ only rely on the alphabet $E$. In particular, they are independent of any Cantor real base $\B\in E^{\mathbb{N}}$ given as input. The following proposition is a direct consequence of the definitions. 

\begin{proposition}
Let $E\subset \mathbb{R}_{>1}$ be a finite alphabet and let $\B \in E^{\mathbb{N}}$ be a Cantor real base. The output of the transducer $\mathcal{T}_E$ (resp.\ $\mathcal{T}^*_E$) when reading $\B$ from a state $r\in [0,1]$ is the $\B$-expansion (resp.\ quasi-greedy $\B$-expansion) of~$r$.
\end{proposition}

\begin{proof}
    States of the transducer simply represent the memory of the considered intermediate values obtained when performing the (quasi-)greedy algorithm and the labels refer to the reading of the input base and the corresponding output digit. More formally, consider $r\in[0,1]$. Write $\B=(\beta_n)_{n\in\mathbb{N}}$, $d_{\B}(r)=(a_n)_{n\in\mathbb{N}}$ and $d^*_{\B}(r)=(b_n)_{n\in\mathbb{N}}$. Then $a_0=\lfloor \beta_0 \rfloor$, $b_0=\lceil \beta_0 -1 \rceil$, and for all $n\ge 1$, the digits $a_n$ and $b_n$ are computed as 
    \[
        a_n=\lfloor \beta_n\mathsf{T}_{\beta_{n-1}}\cdots\mathsf{T}_{\beta_0}(r)\rfloor\quad \text{ and }\quad
        b_n=\lceil  \beta_n\mathsf{T}_{\beta_{n-1}}\cdots\mathsf{T}_{\beta_0}(r)-1\rceil.
    \]
\end{proof}

\begin{definition}
  When considering a given $r\in [0,1]$ as initial state, the part of $\mathcal{T}_E$ (resp.\ $\mathcal{T}^*_E$) that can be reached from~$r$ is a transducer that we denote by~$\mathcal{T}_{E,r}$ (resp.~$\mathcal{T}^*_{E,r}$).
\end{definition}

\begin{proposition}
    \label{prop:TE-TE*-simultanement-finis}
    Let $E\subset \mathbb{R}_{>1}$ be a finite alphabet and let $r\in[0,1]$. The transducers $\mathcal{T}_{E,r}$ and $\mathcal{T}^*_{E,r}$ are either both finite or both infinite.
    \end{proposition}

\begin{proof}
    First, observe that if the two transitions starting from the state $r$ and corresponding to a given input $\beta$ in $\mathcal{T}_E$ and $\mathcal{T}^*_E$ do not coincide, then it means that $\beta r$ is a positive integer, $\mathsf{T}_{\beta}(r)=0$ and $\mathsf{T}^*_{\beta}(r)=1$. 
   Now, if infinitely many states are accessible from $r$ in  $\mathcal{T}_E$ (resp.\ $\mathcal{T}^*_E$) then there must exist a path in $\mathcal{T}_E$ (resp.\ $\mathcal{T}^*_E$) visiting infinitely many distinct states but not visiting the state $0$ (resp.\ the state $1$). Thus, these paths must be present in both $\mathcal{T}_E$ and $\mathcal{T}^*_E$. Hence the conclusion.
\end{proof}

Our aim in this section is to provide a family of Cantor real bases for which the transducers $\mathcal{T}_{E,r}$ and $\mathcal{T}^*_{E,r}$ are finite. In view of the previous result, we will present the proofs of our subsequent results using only the greedy transducer $\mathcal{T}_{E,r}$. 

The following norm is also called (discrete) supremum norm, Chebyshev norm or max norm when the supremum is in fact a maximum. 

\begin{definition}[Max norm]
\label{def:norm}
Let $d\ge 1$ and let $z=(z_1,\ldots,z_d)$ be a $d$-tuple of pairwise distinct complex numbers.
On the finite dimensional vector space $\mathbb{R}_{<d}[X]$ of polynomials of degree less than~$d$, we define a norm by 
\[\lvert \lvert P\rvert \rvert _{z}=\max_{1\le i\le d}|P(z_i)|.\]
Verifying that it is a norm is straightforward. Note that positive definiteness follows from the fundamental theorem of algebra.
\end{definition}

For an algebraic number $\delta$ of degree $d$, we let $\mathcal{M}_\delta$ denote its minimal polynomial. The field extension $\mathbb{Q}(\delta)$ is isomorphic to $\mathbb{Q}[X]/\langle \mathcal{M}_\delta\rangle$ which can be identified with the set $\mathbb{Q}_{<d}[\delta]$ of polynomials in $\delta$ of degree less than $d$ where the operations are performed modulo $\mathcal{M}_\delta$. If moreover $\delta$ is an algebraic integer, then there is also a ring isomorphism between $\mathbb{Z}[\delta]$ and $\mathbb{Z}[X]/\langle \mathcal{M}_\delta\rangle$ which can be identified with $\mathbb{Z}_{<d}[\delta]$. Indeed, in this case, the minimal polynomial $\mathcal{M}_\delta$ is a monic polynomial with integer coefficients. By Euclidean division, any polynomial $P\in\mathbb{Z}[X]$ can be uniquely written as $P=Q\cdot \mathcal{M}_\delta+R$ for some $R\in \mathbb{Z}_{<d}[\delta]$. If $\delta=\delta_1,\delta_2,\ldots,\delta_d$ are the Galois conjugates of $\delta$, i.e., the roots of the minimal polynomial $\mathcal{M}_\delta$, then, for all $i\in\{1,\ldots,d\}$, evaluation at $\delta_i$ gives
$P(\delta_i)=R(\delta_i)$. This explains one of the assumption of the forthcoming \cref{thm:main_transducer}.

Moreover $\mathbb{Z}_{<d}[X]$ is a subset of the vector space $\mathbb{R}_{<d}[X]$. So it can be equipped with a norm as in~\cref{def:norm} for any choice of $d$ complex numbers. Nevertheless, because of the equality $P(\delta_i)=R(\delta_i)$, we will consider the norm associated with the particular $d$-tuple $(\delta_1,\ldots,\delta_d)$.

\begin{example}
    Let $\delta=\delta_1$ be the golden mean and $\delta_2=(1-\sqrt{5})/2$ be its conjugate. In~\cref{fig:norm}, the point of coordinate $(a,b)$ represents the polynomial $aX+b$ and the color of this point is given by the max norm of the polynomial corresponding to the pair $z=(\delta_1,\delta_2)$. 

    \begin{figure}[ht]
        \centering
        \includegraphics[width=0.4\linewidth]{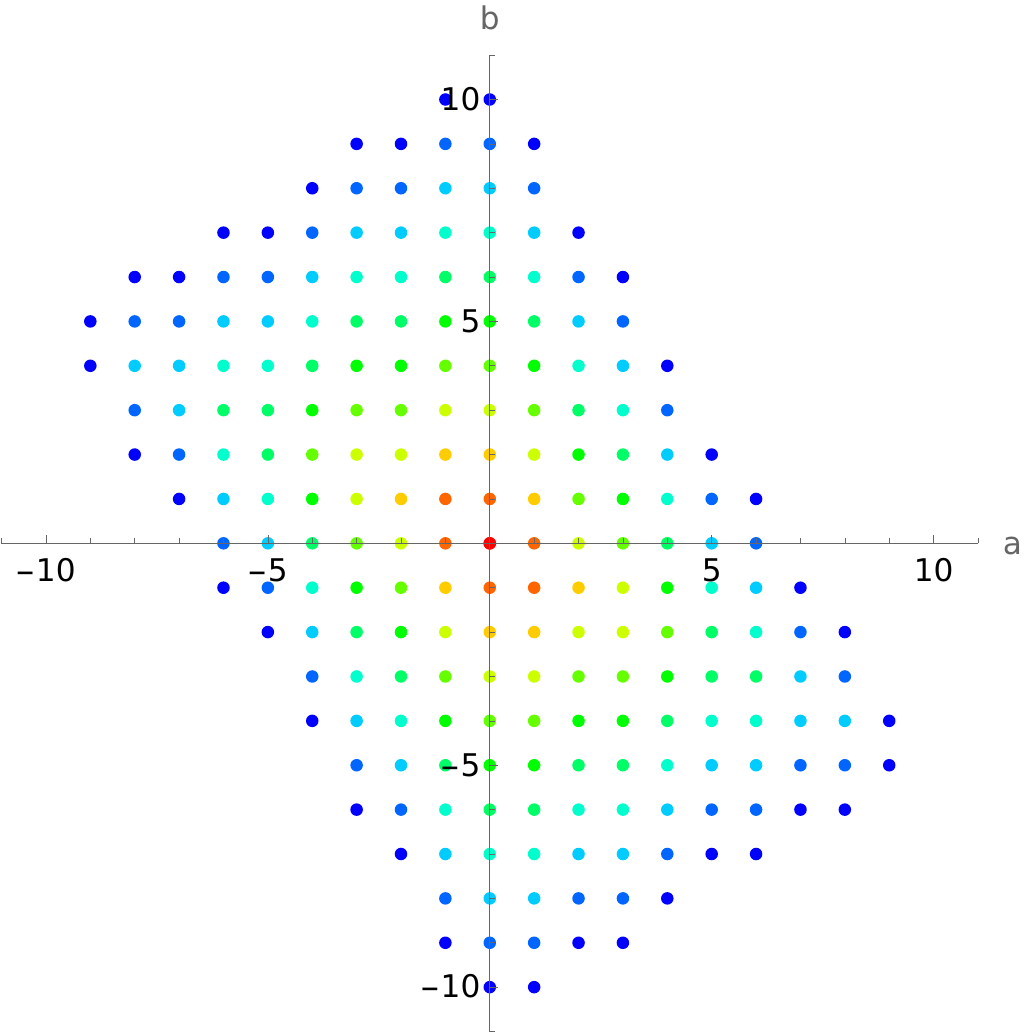}
        \caption{Max norm of $aX+b$ for the golden mean.}
        \label{fig:norm}
    \end{figure}

    We have depicted polynomials from norm $0$ (in red) to norm $10$ (in blue). The number of polynomials in $\mathbb{Z}_{<2}[X]$ of max norm $i$ for $i$ ranging from $0$ to $10$ is $1, 6, 8, 14, 16, 18, 24, 26, 32, 34, 38$. Estimate on the number of polynomials in $\mathbb{Z}[X]$ of degree less than~$d$ will provide an upper bound on the number of states of the transducers that we shall encounter. Minkowski’s theorem on lattice point counting can be used, for instance, to obtain such an estimate~\cite{Matousek}.
\end{example}

\begin{lemma}\label{lem:minpolydelta}
    Let $\delta$ be an algebraic number of degree $d$ with conjugates $\delta_1,\delta_2,\ldots,\delta_d$ and
    let $G\in \mathbb{Q}[X]$. The number $G(\delta)$ is algebraic of degree at most~$d$ and its conjugates belong to $\{G(\delta_1),\ldots,G(\delta_d)\}$. In particular, if $G(\delta)$ has degree~$d$, its conjugates are exactly $G(\delta_1),\ldots,G(\delta_d)$ with no repetition.
\end{lemma}

\begin{proof}
Let $M$ be the splitting field of the minimal polynomial of $\delta$.
For all conjugates $\gamma$ of $G(\delta)$, there is a $\mathbb{Q}$-automorphism $\psi$ of $M$ that sends $G(\delta)$ to $\gamma$. Since the polynomial $G$ has rational coefficients, one has $\gamma=\psi(G(\delta))=G(\psi(\delta))$ and $\psi(\delta)$ must be a conjugate of $\delta$. This shows that $\gamma$ has the expected form. 

For an arbitrary polynomial, merge may occur: there could exist two distinct conjugates $\delta_i$ and $\delta_j$ such $G(\psi(\delta_i))=G(\psi(\delta_j))$. However, if $G(\delta)$ has degree~$d$, it has $d$ pairwise conjugates and no merge occur.
\end{proof}

Even though the transducer $\mathcal{T}_E$ is infinite, under suitable assumptions, the next result shows how to extract a finite transducer which produces $d_\B(r)$ when fed with any Cantor real base $\B\in E^\mathbb{N}$.

\begin{theorem}\label{thm:main_transducer}
  Let $\delta$ be an algebraic integer of degree $d$
  and let $E$ be a finite alphabet of Pisot numbers of degree $d$ belonging to $\mathbb{Z}[\delta]$. For all $r\in\mathbb{Q}(\delta)\cap [0,1]$, the transducers $\mathcal{T}_{E,r}$ and $\mathcal{T}^*_{E,r}$ are finite.
\end{theorem}

\begin{proof}
First note that, thanks to our assumption that $E\subset \mathbb{Z}[\delta]$, starting from a state in $\mathbb{Q}(\delta)\cap[0,1]$ in the transducer $\mathcal{T}_E$, one can only visit states belonging to the same set. 

Let $r\in\mathbb{Q}(\delta)\cap [0,1]$ and let $s$ be a state that can be reached from $r$ with a path of length~$\ell$ labeled by $\beta_0\cdots \beta_{\ell-1}$ with $\beta_j$ in $E$ for all $j$. For the sake of notation, we let $\mathsf{T}_j$ be the $\beta_j$-transformation. 
Hence 
\[
    s=\mathsf{T}_{\ell-1}  \cdots \mathsf{T}_0(r),
\]
which can be rewritten as 
\begin{align} \label{eq:Maininduction}
    s= \beta_{\ell-1} \cdots \beta_0\, r - \sum_{m=0}^{\ell-1}  \beta_{\ell-1} \cdots \beta_{m+1}
\lfloor \beta_m \mathsf{T}_{m-1}\cdots \mathsf{T}_0(r)\rfloor.
\end{align} 

Let us prove \eqref{eq:Maininduction} by induction on $\ell$. The identity is trivially verified for $\ell=0$. Let $\ell\geq 0$ such that \eqref{eq:Maininduction} is satisfied for $s=\mathsf{T}_{\ell-1}  \cdots \mathsf{T}_0(r)$ and let $s'=\mathsf{T}_{\ell} \cdots \mathsf{T}_0(r)$. Then
\begin{align*}
    s'&=\beta_{\ell}\mathsf{T}_{\ell-1}  \cdots \mathsf{T}_0(r) - \lfloor \beta_\ell \mathsf{T}_{\ell-1}  \cdots \mathsf{T}_0(r) \rfloor, \\
    &=\beta_\ell\beta_{\ell-1} \cdots \beta_0 r - \sum_{m=0}^{\ell-1}\beta_\ell \cdots \beta_{m+1}
    \lfloor \beta_m \mathsf{T}_{m-1}\cdots \mathsf{T}_0(r)\rfloor - \lfloor \beta_\ell \mathsf{T}_{\ell-1}  \cdots \mathsf{T}_0(r) \rfloor,\\
    & =\beta_\ell \cdots \beta_0 r - \sum_{m=0}^{\ell}\beta_\ell \cdots \beta_{m+1}
    \lfloor \beta_m \mathsf{T}_{m-1}\cdots \mathsf{T}_0(r)\rfloor,
\end{align*} 
where the second equality follows from the recurrence hypothesis.

For all $\beta\in E$, let $G_\beta\in\mathbb{Z}_{<d}[X]$ be such that $\beta=G_\beta(\delta)$, and let $R\in\mathbb{Q}_{<d}[X]$ be such that $r=R(\delta)$. 
Let us write 
\[
    R=\sum_{0\le i < d}r_iX^{i}
    \text{ with }r_i=\frac{p_i}{q_i} \in \mathbb{Q}
\]
and define $q=\text{lcm}\{q_0,\ldots,q_{d-1}\}$. For simplicity, we simply write $G_j$ for the polynomial $G_{\beta_j}$, where $\beta_0\cdots\beta_{\ell-1}$ is the label of the path that leads us from $r$ to $s$. Now, let $S\in\mathbb{Q}[X]$ be the polynomial 
\begin{equation}
  \label{eq:S}
  S=G_{\ell-1} \cdots G_0 R-\sum_{m=0}^{\ell-1}  c_m \, G_{\ell-1} \cdots G_{m+1}
\end{equation}
where the coefficients $c_m:=\lfloor \beta_m \mathsf{T}_{m-1}\cdots \mathsf{T}_0(r)\rfloor$ are integers bounded by $\max E$. From~\eqref{eq:Maininduction}, we infer that $s=S(\delta)$. 

Our aim is to show that there are finitely many such possible values $S(\delta)$, independently of the length $\ell$ of the path from $r$ to $s$. This will in turn imply that there are finitely many states $s$ that can be reached from $r$ in $\mathcal{T}_E$. Note that the degree of the polynomial $S$ is not bounded independently of $\ell$. Inspired by~\cite[Prop.~2.3.33]{CANT}, our strategy is to make use of a well-chosen max norm on the discrete lattice $\mathbb{Z}_{<d}[\delta]$.

By choice of $q$, the polynomial $qS$ has integer coefficients. By Euclidean division of $qS$ by the minimal polynomial~$\mathcal{M}_\delta$ of $\delta$, we get $qS=Q\cdot \mathcal{M}_\delta+S'$ with $S'\in \mathbb{Z}_{<d}[X]$.
Let $\delta_1=\delta,\delta_2, \ldots, \delta_d$ be the roots of $\mathcal{M}_\delta$. 
On the discrete lattice $\mathbb{Z}_{<d}[\delta]$ of rank~$d$, we use the max norm given in~\cref{def:norm} associated with the $d$-tuple $(\delta_1,\ldots,\delta_d)$, which we simply denote here by $\lvert\lvert\cdot\lvert\lvert$, i.e.,
\[
    \lvert \lvert P\rvert \rvert =\max_{1\le i\le d}|P(\delta_i)|.
\]
Our aim is to prove that $\lvert \lvert S'\rvert \rvert$ is bounded independently of $\ell$. Since $S(\delta)=\frac{1}{q}S'(\delta)$,
this will be enough to show that there are finitely many possible states accessible from $r$ in $\mathcal{T}_E$. Since \[
      \lvert \lvert S'\rvert \rvert =\max_{1\le i\le d}|qS(\delta_i)|,
\] in order to get a bound on $\lvert \lvert S'\rvert \rvert $, it is sufficient to get a bound on $|S(\delta_i)|$ for every $i\in\{1,\ldots,d\}$. First, we know $S(\delta_1)=S(\delta)=s\le 1$.

We now turn to the non-trivial conjugates of $\delta$. By assumption, all elements $\beta$ in $E$ are Pisot numbers of degree $d$. We know from~\cref{lem:minpolydelta} that the conjugates of every $\beta$ in $E$ are exactly given by $\beta=G_\beta(\delta_1),G_\beta(\delta_2),\ldots,G_\beta(\delta_d)$. We then get from the Pisot condition that for all $i\in\{2,\ldots,d\}$, we have $|G_\beta(\delta_i)|<1$. Thus we get
\[
    M_i:=\max_{\beta\in E} |G_\beta(\delta_i)|<1
\] 
for all $i\in\{2,\ldots,d\}$. Therefore, and by using~\eqref{eq:S}, we get that for all $i\in\{2,\ldots,d\}$, 
\begin{align*}
  | S(\delta_i)|
        &=\left| G_{\ell-1}(\delta_i) \cdots G_0(\delta_i) R(\delta_i)
        -\sum_{m=0}^{\ell-1} c_m  G_{\ell-1}(\delta_i) \cdots G_{m+1}(\delta_i)\right|\\
        &\le |R(\delta_i)|+ \max E \sum_{m=0}^{\ell-1}  M_i^{\ell-1-m}, \\
        &=|R(\delta_i)|+\max E \sum_{m=0}^{\ell-1}  M_i^m, \\ 
        &\le |R(\delta_i)|+  \frac{\max E}{1-M_i}.
\end{align*}
This bound only depends on the initial element $r$ (providing the polynomial~$R$) and the algebraic integer $\delta$ (providing the conjugates~$\delta_i$). 
\end{proof}

We recall some well-known facts on Pisot numbers. For the sake of completeness, we sketch a proof.

\begin{lemma}
\label{lem:Pisot}
    Every real algebraic number field $K$ of degree $d$ contains a Pisot number of degree $d$. Moreover, the set of all Pisot numbers of degree $d$ in $K$ is closed under multiplication.
\end{lemma}

\begin{proof}
    The first part of this result can be found in \cite[Thm.~2]{Salem} (it relies on Minkowski's theorem on linear forms). 

    Let $\alpha,\beta$ be two Pisot numbers of degree~$d$ belonging to $K$ and let us prove that $\gamma=\alpha\beta$ is also a Pisot number. We have $K=\mathbb{Q}(\alpha)=\mathbb{Q}(\beta)$, hence $\beta=P(\alpha)$ where $P\in\mathbb{Q}[X]$. Thus, if the conjugates of $\alpha$ are $\alpha_1=\alpha,\alpha_2\ldots,\alpha_d$, then by~\cref{lem:minpolydelta}, the conjugates of $\beta$ are exactly given by $P(\alpha_1)=\beta,P(\alpha_2),\ldots,P(\alpha_d)$ (since $\beta$ is of degree $d$) and the conjugates of $\gamma$ are among $\alpha_1P(\alpha_1)=\gamma,\alpha_2P(\alpha_2),\ldots,\alpha_dP(\alpha_d)$. Clearly, $\gamma>1$. Since $|\alpha_iP(\alpha_i)|<1$ for $i\ne 1$ and since $\gamma$ is an algebraic integer (the algebraic integers form a ring), we get that $\gamma$ is a Pisot number. Let us argue that $\gamma$ still has degree $d$. The polynomial 
    \[
        G:=(X-\alpha_1P(\alpha_1))\cdots(X-\alpha_dP(\alpha_d))
    \]
    has rational coefficients (since each of its coefficients can be expressed as a symmetric polynomial in $\alpha_1,\ldots,\alpha_n$ with rational coefficients, and since the polynomial $(X-\alpha_1)\cdots(X-\alpha_n)$ has rational --- and even integral --- coefficients). Since all its roots $\alpha_iP(\alpha_i)$ are algebraic integers, we infer that $G$ has in fact integer coefficients. But then, it must be irreducible by Kronecker's theorem. Therefore, $G$ is the minimal polynomial of $\gamma$ proving that $\gamma$ has degree $d$.
\end{proof}

The following statement is the special case where all bases are powers of the same Pisot number. 

\begin{corollary}
\label{cor:pisotpowers}
  Let $E$ be a finite set of positive powers of a Pisot number~$\delta$.  
  For all $r\in\mathbb{Q}[\delta]\cap [0,1]$, the transducers $\mathcal{T}_{E,r}$ and $\mathcal{T}^*_{E,r}$ are finite. 
\end{corollary}

\begin{proof}
    From~\cref{lem:Pisot}, we know that powers of a Pisot number are Pisot numbers with the same degree. Thus, \cref{thm:main_transducer} can be applied.
\end{proof}

Automatic sequences and morphic sequences are closed under $1$-uniform transduction. For instance, see~\cite[Theorem 6.9.2]{AS03} and~\cite{Dekking1994}. The next statement is therefore straightforward.

\begin{corollary}
\label{cor:automatic}
 Let $\delta$ be an algebraic integer of degree $d$ and let $E$ be a finite alphabet of Pisot numbers of degree $d$ belonging to $\mathbb{Z}[\delta]$. If $\B\in E^\mathbb{N}$ is a $b$-automatic sequence for some integer base $b\ge 2$ (resp.\ a morphic sequence), then for all $r\in\mathbb{Q}(\delta)\cap [0,1]$, the sequences $d_\B(r)$ and $d_\B^*(r)$ are  also $b$-automatic (resp.\ morphic).
\end{corollary}


\subsection{Building a few examples}

Under the assumptions of~\cref{thm:main_transducer}, both transducers $\mathcal{T}_{E,r}$ and $\mathcal{T}^*_{E,r}$ are finite: \cref{fig:algo-transducteur-glouton} and \cref{fig:algo-transducteur-quasi-glouton} terminate and produce transducers $\mathcal{T}_{E,r}$ and $\mathcal{T}^*_{E,r}$, respectively. These algorithms explore the state-space incrementally, propagating outward --- much like an oil stain spreading on a surface. Such algorithms can be easily implemented in any computer algebra system. The examples given in this paper were computed using {\tt Mathematica} because when dealing with algebraic numbers, we were able to work with exact arithmetic. Indeed, relying on numerical approximations could lead to distinguishing states that actually correspond to the same real number. An implementation of the two algorithms is available on \url{https://hdl.handle.net/2268/333750}, as well as all the computed transducers from in the paper.

\begin{algorithm}   
 \begin{algorithmic}[1] 
        \Procedure{Greedy Transducer}{E,r} 
            \State $Q\gets \emptyset$ \Comment{Set of states of the transducer}
            \State $L\gets \{r\}$ \Comment{List of states to be processed}
            \While{$L\not=\emptyset$}
            \State Let $q$ be the first element in $L$
            \State Remove $q$ from $L$ and add it to $Q$
                \For{$\beta$ in $E$}
                \If {$\mathsf{T}_{\beta}(q)\not\in Q\cup L$} add $q$ to $L$ \Comment{Detection of a new state}
                \EndIf
                \State add a transition $(q,\beta\,|\,\lfloor \beta q\rfloor,\mathsf{T}_{\beta}(q))$ \Comment{Defining transitions}
                \EndFor
            \EndWhile
        \EndProcedure
\end{algorithmic}
    \caption{Algorithm producing $\mathcal{T}_{E,r}$}
    \label{fig:algo-transducteur-glouton}
\end{algorithm}

\begin{algorithm}     
 \begin{algorithmic}[1] 
        \Procedure{Quasi-greedy Transducer}{E,r} 
            \State $Q\gets \emptyset$ \Comment{Set of states of the transducer}
            \State $L\gets \{r\}$ \Comment{List of states to be processed}
            \While{$L\not=\emptyset$}
            \State Let $q$ be the first element in $L$
            \State Remove $q$ from $L$ and add it to $Q$
                \For{$\beta$ in $E$}
                \If {$\mathsf{T}^*_{\beta}(q)\not\in Q\cup L$} add $q$ to $L$ \Comment{Detection of a new state}
                \EndIf
                \State add a transition $(q,\beta\,|\,\lceil \beta q-1\rceil,\mathsf{T}^*_{\beta}(q))$ \Comment{Defining transitions}
                \EndFor
            \EndWhile
        \EndProcedure
\end{algorithmic}
    \caption{Algorithm producing $\mathcal{T}^*_{E,r}$}
    \label{fig:algo-transducteur-quasi-glouton}
\end{algorithm}

\begin{example}
\label{exa:continue}
For a set $E$ of finitely many integer bases, the assumptions of~\cref{thm:main_transducer} are satisfied. For instance, with $E=\{2,3\}$, we have the transducers $\mathcal{T}^*_{E,1}$ and $\mathcal{T}^*_{E,\frac{1}{5}}$ depicted in~\cref{fig:exa0}. 

\begin{figure}[ht] 
\begin{center}
\begin{tikzpicture}[->,thick,main node/.style={circle,draw,minimum size=12pt,inner sep=2pt}]
\tikzstyle{every node}=[shape=circle,draw=red,minimum size=12pt,inner sep=2pt]

\node(a0)[color=red] at (-2,0) {{\color{black} 1}};
\node(a1)[color=red] at (2,0) {{\color{black} 1}};
\node(a2)[color=red] at (4,0) {{\color{black} 1/5}};

\tikzstyle{every node}=[shape=circle,fill=none,draw=black,minimum size=12pt,inner sep=2pt]

\node(b01) at (0,0) {0};

\node(b21) at (6,2) {2/5};
\node(b22) at (8,0) {4/5};
\node(b23) at (6,-2) {3/5};

\tikzstyle{every path}=[color =black, line width = 0.5 pt]
\tikzstyle{every node}=[shape=rectangle,draw=none,minimum size=10pt,inner sep=2pt]

\draw [->] (a0) to [bend left=20] node [above] {\small $2|2$}  (b01);
\draw [->] (a0) to [bend right=20] node [below] {\small $3|3$} (b01);

\draw [->] (a2) to [bend left=15] node [above left] {\small $2|0$}  (b21);
\draw [->] (a2) to [bend left=15] node [above right] {\small $3|0$} (b23);

\draw [->] (b21) to [bend left=15] node [above right] {\small $2|0$}  (b22);
\draw [->] (b21) to [bend left=15] node [below right] {\small $3|1$} (a2);

\draw [->] (b22) to [bend left=15] node [below right] {\small $2|1$}  (b23);
\draw [->] (b22) to [bend left=15] node [below left] {\small $3|2$} (b21);

\draw [->] (b23) to [bend left=15] node [below left] {\small $2|1$}  (a2);
\draw [->] (b23) to [bend left=15] node [above left] {\small $3|1$} (b22);

\tikzset{every loop/.style={min distance=10mm,in=60,out=120,looseness=10}}

\draw [->] (b01) to [loop right] node [above] {\small $2|0$} ();
\draw [->] (a1) to [loop right] node [above] {\small $2|1$} ();
\tikzset{every loop/.style={min distance=10mm,in=240,out=300,looseness=10}}

\draw [->] (b01) to [loop left] node [below] {\small $3|0$} ();
\draw [->] (a1) to [loop left] node [below] {\small $3|2$} ();

\end{tikzpicture}
\end{center}
\caption{The transducers $\mathcal{T}_{E,1}$ (left), $\mathcal{T}^*_{E,1}$ (middle) and $\mathcal{T}_{E,\frac{1}{5}}=\mathcal{T}^*_{E,\frac{1}{5}}$ (right) for $E=\{2,3\}$.}
 \label{fig:exa0}
\end{figure}
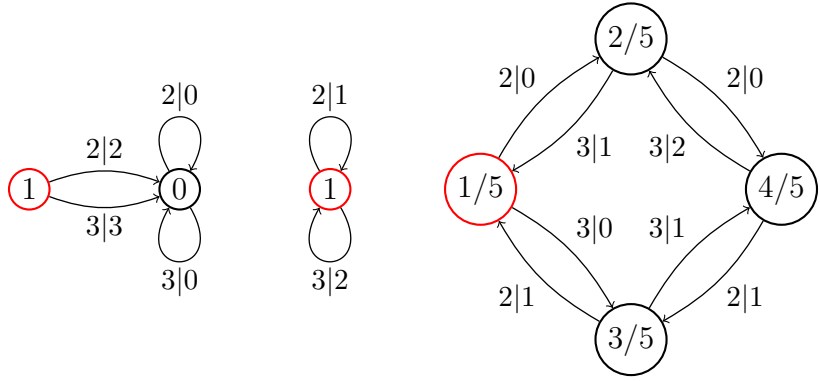

  If we consider the rational number $r=932/3885$, we get a transducer $\mathcal{T}_{E,r}$ with $180$ states (in this case, the greedy and quasi-greedy transducers coincide). This number of states is much larger than in~\cref{exa:first_transducer} but here we allow any sequence of $2$ and $3$ as Cantor base while in~\cref{sec:Cantorbaseintegers}, the only input Cantor bases were made of blocks of $23$ and $32$. In \cref{fig:180}, we have represented the accessible part of the $180$-state transducer that can be reached from the initial state but restricting the input to sequences made of $23$ and $32$. 

\begin{figure}[ht] 
\begin{center}
\begin{tikzpicture}[->,thick,main node/.style={circle,draw,minimum size=12pt,inner sep=2pt}]
\tikzstyle{every node}=[shape=circle,draw=red,minimum size=12pt,inner sep=2pt]

\node(a)[color=red] at (0,0) {};

\tikzstyle{every node}=[shape=circle,fill=none,draw=black,minimum size=12pt,inner sep=2pt]

\node(b1) at (3,0) {};
\node(b2) at (0,3) {};
\node(b3) at (3,3) {};

\node(c1) at (4.5,1.5) {};
\node(c2) at (7.5,1.5) {};
\node(c3) at (4.5,4.5) {};
\node(c4) at (7.5,4.5) {};

\node(d1) at (3,6) {};
\node(d2) at (6,6) {};
\node(d3) at (9,6) {};
\node(d4) at (9,3) {};
\node(d5) at (9,0) {};
\node(d6) at (6,0) {};

\tikzstyle{every path}=[color =black, line width = 0.5 pt]
\tikzstyle{every node}=[shape=rectangle,draw=none,minimum size=10pt,inner sep=2pt]

\draw [->] (a) to [] node [left] {\small $2|0$}  (b2);
\draw [->] (a) to [] node [below] {\small $3|0$}  (b1);

\draw [->] (b2) to [] node [above] {\small $3|1$}  (b3);
\draw [->] (b1) to [] node [left] {\small $2|1$}  (b3);

\draw [->] (b3) to [] node [below right] {\small $2|0$}  (c3);
\draw [->] (b3) to [] node [left] {\small $3|1$}  (d1);

\draw [->] (c3) to [] node [below right] {\small $3|2$}  (d2);
\draw [->] (d1) to [] node [above] {\small $2|0$}  (d2);

\draw [->] (d2) to [] node [above] {\small $2|1$}  (d3);
\draw [->] (d2) to [] node [below left] {\small $3|1$}  (c4);

\draw [->] (d3) to [] node [right] {\small $3|0$}  (d4);
\draw [->] (c4) to [] node [below left] {\small $2|1$}  (d4);

\draw [->] (d4) to [] node [above left] {\small $2|1$}  (c2);
\draw [->] (d4) to [] node [right] {\small $3|2$}  (d5);

\draw [->] (c2) to [] node [above left] {\small $3|1$}  (d6);
\draw [->] (d5) to [] node [below] {\small $2|0$}  (d6);

\draw [->] (d6) to [] node [above right] {\small $2|1$}  (c1);
\draw [->] (d6) to [] node [below] {\small $3|2$}  (b1);

\draw [->] (c1) to [] node [above right] {\small $3|2$}  (b3);

\end{tikzpicture}
\end{center}
\caption{Part of the transducer $\mathcal{T}_{E,932/3885}$ with inputs in $\{23,32\}^*$.}
\label{fig:180}
\end{figure}
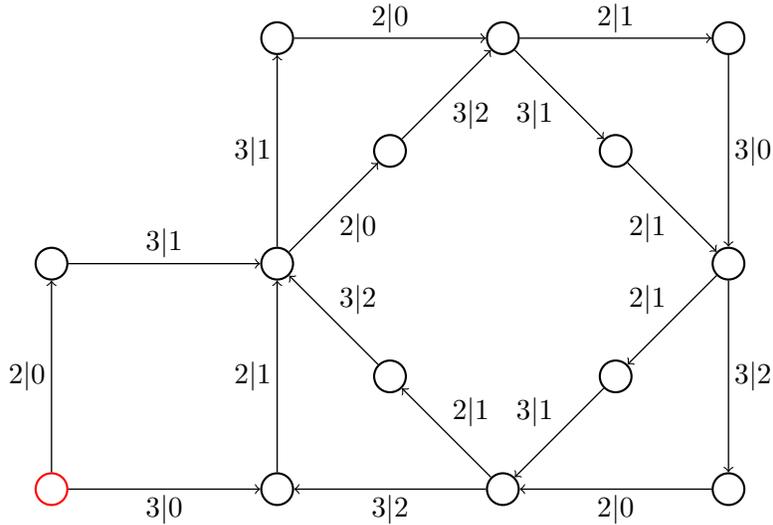

This subtransducer almost coincides with the construction discussed at the end of~\cref{exa:first_transducer}. Indeed, in the latter example, each edge label $a|h_a(j)$  in \cref{fig:transducer-fig} gave rise to two edges going through a new intermediate state. This construction thus results in a transducer with $15$ states (the $5$ original states and the $2\cdot 5$ new intermediate states). However, when using the algorithm computing $\mathcal{T}_{E,r}$, two states in the construction correspond to the same remainder and are therefore merged here resulting in $14$ states.
\end{example}

\begin{remark}\label{rem:complexity}
    Under the assumptions of~\cref{thm:main_transducer}, for a fixed $r$, the transducer $\mathcal{T}_{E,r}$ (or $\mathcal{T}^*_{E,r}$) can be used to derive a measure of complexity of Cantor bases~$\B$ as the ratio of the number of states visited in $\mathcal{T}_{E,r}$ with input $\B$ compared with the total number of states of this transducer. One can expect that a random sequence will visit the full state space while specific Cantor bases will only visit a subset of the state space. In the previous example, we could say that a Cantor base in $\{23,32\}^\omega$ has complexity (with respect to $r=932/3885$) equal to $14/180<1/12$. 
    
    Such an observation leads to another question: the set of visited states for an input~$\B$ depends on the structure of the transducer $\mathcal{T}_{E,r}$ and in particular, its connectedness (see \cref{sec:connectedness}).
\end{remark}
    
\begin{example}
  Let $\varphi=\frac{1+\sqrt{5}}{2}$ be the golden mean and let $E=\{\varphi,\varphi^3\}$. The assumptions of~\cref{thm:main_transducer} are satisfied (also see Corollary~\ref{cor:pisotpowers}). The transducers $\mathcal{T}^*_{E,1}$ and $\mathcal{T}^*_{E,\frac{1}{2}}$ are finite and depicted in~\cref{fig:exa1}. 

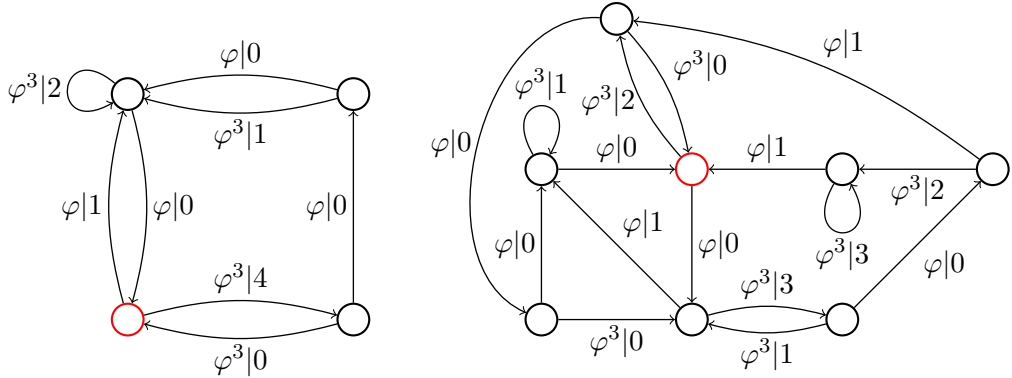
\begin{figure}[ht] 
\begin{center}
\begin{tikzpicture}[->,thick,main node/.style={circle,draw,minimum size=12pt,inner sep=2pt}]
\tikzstyle{every node}=[shape=circle,draw=red,minimum size=12pt,inner sep=2pt]

\node(a0)[color=red] at (-1.5,0) {};
\node(a1)[color=red] at (6,2) {};

\tikzstyle{every node}=[shape=circle,fill=none,draw=black,minimum size=12pt,inner sep=2pt]

\node(b01) at (-1.5,3) {};
\node(b02) at (1.5,0) {};
\node(b03) at (1.5,3) {};

\node(b11) at (4,0) {};
\node(b12) at (6,0) {};
\node(b13) at (8,0) {};
\node(b14) at (4,2) {};
\node(b15) at (8,2) {};
\node(b16) at (5,4) {};
\node(b17) at (10,2) {};

\tikzstyle{every path}=[color =black, line width = 0.5 pt]
\tikzstyle{every node}=[shape=rectangle,draw=none,minimum size=10pt,inner sep=2pt]
\tikzset{every loop/.style={min distance=10mm,in=210,out=140,looseness=10}}

\draw [->] (a0) to [bend left=15] node [left] {\small $\varphi|1$}  (b01);
\draw [->] (a0) to [bend left=15] node [above] {\small $\varphi^3|4$} (b02);

\draw [->] (b01) to [bend left=15] node [right] {\small $\varphi|0$}  (a0);
\draw [->] (b01) to [loop] node [left] {\small $\varphi^3|2$} ();

\draw [->] (b02) to [] node [left] {\small $\varphi|0$}  (b03);
\draw [->] (b02) to [bend left=15] node [below] {\small $\varphi^3|0$} (a0);

\draw [->] (b03) to [bend right=15] node [above] {\small $\varphi|0$}  (b01);
\draw [->] (b03) to [bend left=15] node [below] {\small $\varphi^3|1$} (b01);

\draw [->] (a1) to [bend left=15] node [left] {\small $\varphi^3|2$}  (b16);
\draw [->] (a1) to [] node [right] {\small $\varphi|0$} (b12);

\draw [->] (b11) to [] node [left] {\small $\varphi|0$}  (b14);
\draw [->] (b11) to [] node [below] {\small $\varphi^3|0$} (b12);

\draw [->] (b12) to [] node [above right] {\small $\varphi|1$}  (b14);
\draw [->] (b12) to [bend left=15] node [above] {\small $\varphi^3|3$} (b13);

\draw [->] (b13) to [] node [below right] {\small $\varphi|0$}  (b17);
\draw [->] (b13) to [bend left=15] node [below] {\small $\varphi^3|1$} (b12);

\tikzset{every loop/.style={min distance=10mm,in=60,out=120,looseness=10}}

\draw [->] (b14) to [] node [above] {\small $\varphi|0$}  (a1);
\draw [->] (b14) to [loop] node [above] {\small $\varphi^3|1$} ();

\tikzset{every loop/.style={min distance=10mm,in=300,out=240,looseness=10}}

\draw [->] (b15) to [] node [above] {\small $\varphi|1$}  (a1);
\draw [->] (b15) to [loop] node [below] {\small $\varphi^3|3$} ();

\draw [->] (b16) to [bend right=80] node [left] {\small $\varphi|0$} (b11);
\draw [->] (b16) to [bend left=15] node [above right] {\small $\varphi^3|0$}  (a1);

\draw [->] (b17) to [bend right=15] node [above right] {\small $\varphi|1$} (b16);
\draw [->] (b17) to [] node [below] {\small $\varphi^3|2$}  (b15);

\end{tikzpicture}
\end{center}
  \caption{The transducers $\mathcal{T}^*_{E,1}$ (left) and $\mathcal{T}^*_{E,\frac{1}{2}}$ (right) for $E=\{\varphi,\varphi^3\}$.}
  \label{fig:exa1}
\end{figure}

    As an example, if $\mathbf{T}=\varphi\varphi^3\varphi^3\varphi\varphi^3\varphi\varphi\varphi^3\cdots$ is the Thue--Morse word over $E$, then
  \[
    d_{\mathbf{T}}^*(1)=12204002\cdots
    \quad\text{and}\quad
    d_{\mathbf{T}}^*(1/2)=03111003\cdots. 
  \]
\end{example}

\begin{example}
    Let $\varphi$ be the golden mean and let $E=\{\varphi,2\}$. Observe that $E$ is included in $\mathbb{Z}[\varphi]$. Moreover $\varphi$ has algebraic degree $2$ whereas $2$ has algebraic degree $1$. Thus, all the assumptions of~\cref{thm:main_transducer} are satisfied, except the one concerning the degree of the considered Pisot numbers. With the notation of the proof of the theorem, the polynomial $G_2$ is the constant~$2$, and thus evaluations of $G_2$ at $\varphi$ and its conjugate are both equal to~$2$ and $M_2=2$. 
    
    And now consider the Cantor real base $\B=\varphi2^{\omega}$. In that case, the set of states of the transducer that can be reached from $1$ is infinite. Indeed, after the first step using base $\varphi$, one has to represent $1/\varphi$ using only base $2$. But since this number is irrational, its $2$-expansion is aperiodic and we never encounter the same remainder twice in the greedy algorithm. This simple example shows that the assumption about Pisot numbers with the same degree is necessary. 
\end{example}


\section{Periodic expansions}

Having the transducer $\mathcal{T}_{E,r}$ at our disposal, we are interested in determining the Cantor bases~$\B$ for which $d_\B^*(r)$ is ultimately periodic. For instance, considering~\cref{exa:continue} and \cref{fig:180}, it is not hard to see that the only Cantor bases made of blocks $23$ and $32$ giving an ultimately periodic $\B$-expansion of $932/3885$ are the ultimately alternate ones. This section addresses that general question.

\subsection{Relationship with an extension of a theorem of Schmidt to alternate bases}

For an alternate base $\B=(\beta_0,\ldots,\beta_{p-1})^\omega$, if all rational numbers in $[0,1)$ have periodic $\B$-expansions, then $\beta_0,\ldots,\beta_{p-1}$ all belong to the extension field $\mathbb{Q}(\beta_0\cdots\beta_{p-1})$ and moreover, the product $\beta_0\cdots\beta_{p-1}$ must be either a Pisot or Salem number~\cite{CCK,MP}. When the product $\beta_0\cdots\beta_{p-1}$ is a Pisot number and the bases $\beta_0,\ldots,\beta_{p-1}$ all belong to $\mathbb{Q}(\beta_0\cdots\beta_{p-1})$, the set of points in $[0,1)$ having an ultimately periodic $\B$-expansion is precisely $\mathbb{Q}(\beta_0\cdots\beta_{p-1})\cap [0,1)$. 

A family of length-$2$ alternate bases with the property that all rational numbers in $[0,1)$ have purely periodic $\B$-expansions is given in~\cite{MP}. In our setting, by using our graph-theoretical approach, we can easily show that for any ultimately alternate base $\B$ over a set of bases $E$ satisfying the assumptions of~\cref{thm:main_transducer}, all numbers in $\mathbb{Q}(\delta)\cap [0,1)$ have periodic $\B$-expansions. The novelty here lies in the used technique, since this result is in fact a direct consequence of the above-mentioned extension of Schmidt's theorem obtained in \cite{CCK}. Indeed, given an alternate base $\B=(\beta_0,\ldots,\beta_{p-1})^\omega$, Lemma~\ref{lem:Pisot} implies that the hypotheses of Theorem~\ref{thm:main_transducer} are stronger than asking that the product $\beta_0\cdots\beta_{p-1}$ is a Pisot number and $\beta_0,\ldots,\beta_{p-1}\in\mathbb{Q}(\beta_0\cdots\beta_{p-1})$.

\begin{corollary} 
\label{cor:AlternateBase-implies-Periodicity}
 Let $\delta$ be an algebraic integer of degree $d$ and let $E$ be a finite alphabet of Pisot numbers of degree $d$ belonging to $\mathbb{Z}[\delta]$. 
 If $\B$ is an ultimately alternate base whose elements belong to $E$, then all $r\in\mathbb{Q}(\delta)\cap [0,1]$ have ultimately periodic $\B$-expansion $d_\B(r)$ and quasi-greedy $\B$-expansion $d_\B^*(r)$. 
\end{corollary}
  
\begin{proof}
We do the reasoning with $d_\B(r)$, but the same applies to $d_\B^*(r)$.
Assume that $\B$ has a preperiod of length $t$ and a period of length $p$. Since the transducer $\mathcal{T}_{E,r}$ is finite and the input sequence $\B$ is ultimately periodic, there exist integers $k,\ell$ with $t\le k<\ell$ and $k\equiv \ell\pmod p$ such that reading the prefixes of $\B$ of length $k$ and $\ell$ from $r$ lead to the same state. This shows that $d_\B(r)$ is ultimately periodic of period $\ell-k$. 
\end{proof}


\begin{example}
    Consider the golden mean $\varphi$ which is a Pisot number and $4\varphi+1$ which is not a Pisot number (its conjugate is $3-2\sqrt{5}\simeq -1.47$). The product of the two is a Pisot number (its conjugate is $(13-5\sqrt{5})/2\simeq 0.91$). \cref{thm:main_transducer} does not apply while Schmidt's theorem holds. 

    Let $E=\{\varphi,4\varphi+1\}$ and let $\B\in E^{\mathbb{N}}$ be the alternate base $(\varphi,4\varphi+1)^{\omega}$. We have $d_{\B}(1)=1410^{\omega}$ and $d_{\sigma(\B)}(1)=7051(10)^{\omega}=d^*_{\sigma(\B)}(1)$, so $d^*_{\B}(1)=1407051(10)^{\omega}$. Therefore, we can deduce the finite subtransducer of $\mathcal{T}_{E,1}^*$ given in \cref{fig:sub-transducer-alternate}, which corresponds to the input given by the alternate base $\B$.

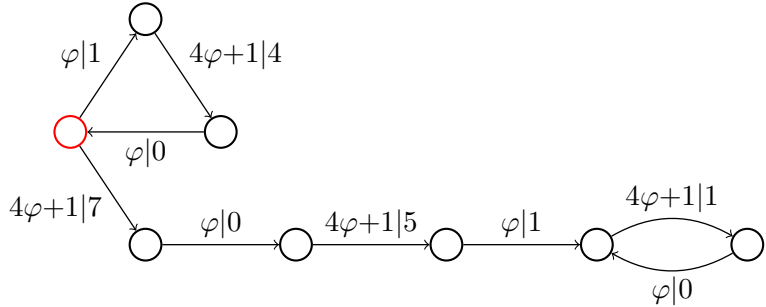
\begin{figure}[ht] 
\begin{center}
\begin{tikzpicture}[->,thick,main node/.style={circle,draw,minimum size=12pt,inner sep=2pt}]
\tikzstyle{every node}=[shape=circle,draw=red,minimum size=12pt,inner sep=2pt]

\node(a)[color=red] at (0,0) {};

\tikzstyle{every node}=[shape=circle,fill=none,draw=black,minimum size=12pt,inner sep=2pt]

\node(b1) at (1,1.5) {};
\node(b2) at (2,0) {};

\node(c1) at (1,-1.5) {};
\node(c2) at (3,-1.5) {};
\node(c3) at (5,-1.5) {};
\node(c4) at (7,-1.5) {};
\node(c5) at (9,-1.5) {};

\tikzstyle{every path}=[color =black, line width = 0.5 pt]
\tikzstyle{every node}=[shape=rectangle,draw=none,minimum size=10pt,inner sep=2pt]

\draw [->] (a) to [] node [above left] {\small $\varphi|1$}  (b1);
\draw [->] (b1) to [] node [above right] {\small $4\varphi{+}1|4$}  (b2);
\draw [->] (b2) to [] node [below] {\small $\varphi|0$}  (a);

\draw [->] (a) to [] node [below left] {\small $4\varphi{+}1|7$}  (c1);
\draw [->] (c1) to [] node [above] {\small $\varphi|0$}  (c2);
\draw [->] (c2) to [] node [above] {\small $4\varphi{+}1|5$}  (c3);
\draw [->] (c3) to [] node [above] {\small $\varphi|1$}  (c4);
\draw [->] (c4) to [bend left] node [above] {\small $4\varphi{+}1|1$}  (c5);
\draw [->] (c5) to [bend left] node [below] {\small $\varphi|0$}  (c4);

\end{tikzpicture}
\end{center}
\caption{Part of the transducer $\mathcal{T}_{E,1}^*$ with inputs in $\{\varphi(4\varphi+1),(4\varphi+1)\varphi\}^*$.}
\label{fig:sub-transducer-alternate}
\end{figure}
    
However, the transducer $\mathcal{T}_{E,1}^*$ is infinite. Indeed, let us recall that every Pisot number $\beta$ is a Parry number, i.e., the $\beta$-expansion of $1$ if finite or ultimately periodic, and that every quadratic Parry number is a Pisot number, see~\cite{Bassino02}. 
Therefore, $4\varphi+1$ is not a Parry number. This means that if $\B=(4\varphi+1)^\omega$, then $d^*_{\B}(1)$ is aperiodic and the transducer $\mathcal{T}_{E,1}^*$ has infinitely many different states. This example shows that if we omit the condition of \cref{thm:main_transducer} that every base is a Pisot number, the result does not hold. 
\end{example}


\subsection{Cantor real bases giving rise to a periodic expansion of a given real number}

We first recall a classical result in combinatorics on words; for instance, see~\cite[Thm.~1.5.3]{AS03}. 

\begin{lemma}[Lyndon--Sch\"utzenberger] \label{lem:LS}
    Let $x,y$ be two non-empty words. The following three conditions are equivalent: 
    \begin{enumerate}
        \item $xy=yx$. 
        \item There exist positive integers $i,j$ such that $x^i=y^j$. 
        \item \label{lem:3} There exist a non-empty word $z$ and positive integers $i,j$ such that $x=z^i$ and $y=z^j$. 
    \end{enumerate}
\end{lemma}

A non-empty word $x$ is \emph{primitive} if it is not an integer power of a shorter word. In the previous lemma, the word $z$ in (\ref{lem:3}) can be taken primitive. 

In what follows, we consider the same notations and assumptions as in~\cref{thm:main_transducer}. From Corollary~\ref{cor:AlternateBase-implies-Periodicity}, we know that for any $r\in\mathbb{Q}(\delta)\cap[0,1]$, the quasi-greedy $\B$-expansion of $r$ is ultimately periodic (or equivalently, its greedy $\B$-expansion if finite or ultimately periodic) whenever $\B$ is an alternate base over the alphabet $E$ of the transducers. The next proposition studies under which conditions on the transducers $\mathcal{T}_{E,r}$ and $\mathcal{T}^*_{E,r}$ the only Cantor real bases over $E$ providing periodic expansions of $r$ are given by the ultimately alternate bases. We will refer to the property described in the statement as the \emph{$2$-walk property}.

\begin{proposition}
\label{pro:2loops}
  Let $\delta$ be an algebraic integer of degree $d$, let $E$ be a finite alphabet of Pisot numbers of degree $d$ belonging to $\mathbb{Z}[\delta]$, and let $r\in\mathbb{Q}(\delta)\cap [0,1]$. The Cantor real bases $\B\in E^\mathbb{N}$ such that $d_{\B}(r)$ is finite or ultimately periodic (resp.\ $d_{\B}^*(r)$ is ultimately periodic) are exactly the ultimately alternate bases if and only if there are no two distinct words $u,v\in E^*$ such that, in the transducer $\mathcal{T}_{E,r}$ (resp.\ $\mathcal{T}^*_{E,r}$), there are two closed walks with inputs $u$ and $v$ respectively, starting from the same state and with the same output.
\end{proposition}

\begin{proof}
As usual, we do the reasoning with $\mathcal{T}_{E,r}$ only, but the same applies to $\mathcal{T}^*_{E,r}$. Assume that there exist two closed walks with distinct inputs $u$ and $v$ starting from some state $s$ with the same output $w$. Since $s$ is accessible from $r$ by construction of $\mathcal{T}_{E,r}$, there exists a path from $r$ to $s$, say with input label $t$. Any input Cantor sequence in $t\{u,v\}^{\omega}$ will provide an ultimately periodic output. Note that $\{u,v\}^{\omega}$ contains (uncountably many) non-periodic sequences.

For the converse, assume that there exists an aperiodic Cantor real base $\B=(\beta_n)_{n\in\mathbb{N}} \in E^{\mathbb{N}}$ such that the corresponding output $d_\B(r)=xyyy\cdots$ is ultimately periodic (the finite case only means that $y=0$). For all $n\ge 0$, we let $r_n$ be the state reached when reading the input $\beta_0\cdots \beta_{n-1}$ from the initial state $r$. 
Since the set of states of $\mathcal{T}_{E,r}$ is finite by~\cref{thm:main_transducer}, when reading the input $\B$, there exists a state that appears infinitely often in the sequence $(r_{|x|+n|y|})_{n\in\mathbb{N}}$. This input $\B$ can be factorized accordingly to highlight all these occurrences as 
\[\B=w_0w_1w_2w_3\cdots,\] 
such that $r_{|w_0|}=r_{|w_0w_1|}=r_{|w_0w_1w_2|}=\cdots$, $|w_0|=|x|+n_0 |y|$ for some $n_0\ge 0$ and for every $j\ge 1$, $|w_j|=n_j|y|$ for some $n_j\ge 1$. In particular, for every $j\ge 1$, the output corresponding to the input $w_j$ starting from $r_{|w_0|}$ is $y^{n_j}$.

We now prove that there exist $i$ and $j$ such that $1\le i<j$ and $w_i^{n_j}\ne w_j^{n_i}$. Proceed by contradiction and assume that for all $i,j\ge 1$ with $i<j$, we have $w_i^{n_j}=w_j^{n_i}$. In particular, $w_1^{n_j}=w_j^{n_1}$ for all $j\ge 2$. By~\cref{lem:LS}, this implies that for all $j\ge 2$, there exist a non-empty primitive word $z_j$ and integers $\ell_j,m_j\ge 1$ such that $w_1=z_j^{\ell_j}$ and $w_j=z_j^{m_j}$. Since $w_1=z_j^{\ell_j}=z_k^{\ell_k}$ with primitive $z_j,z_k$ for all $j,k\ge 2$, \cref{lem:LS} tells us that $z_j=z_k$ for all $j,k\ge 2$, which we denote by~$z$. Therefore every $w_j$ is a power of $z$ and the sequence $\B=w_0z^{\omega}$ is ultimately periodic, a contradiction. Consequently, we have found a state $r_{|w_0|}$ from which start two different closed walks of respective inputs $w_i^{n_j}$ and $w_j^{n_i}$ but with the same output $y^{n_in_j}$. 
\end{proof}

With notation and assumptions of~\cref{thm:main_transducer}, we show that the property given in~\cref{pro:2loops} about having two closed walks with distinct inputs and same output can be decided in $\mathcal{O}((\# Q)^3)$ steps, where $Q$ is the finite set of states of the transducer $\mathcal{T}_{E,r}$ (resp.\ $\mathcal{T}^*_{E,r}$). We now describe the decision procedure inspired by the partition refinement algorithm or Hopcroft's minimization algorithm, a classical result in automata theory. 

\begin{definition} 
Suppose that the hypotheses of~\cref{thm:main_transducer} hold and let $r\in\mathbb{Q}(\delta)\cap[0,1]$. For each state $s$ of $\mathcal{T}_{E,r}$, we recursively define a sequence $(P_{s,k})_{k\ge 1}$ of sets of pairs of states. 
We first define $P_{s,1}$ as the set of pairs of states $(t_1,t_2)$ for which there exist distinct $\beta_1,\beta_2\in E$ and $a\in\mathbb{N}$ such that the transitions
\[
    t_1\stackrel{\beta_1 \,|\, a}{\xrightarrow{\hspace{1cm}}} s  
    \quad \text{ and }\quad
    t_2\stackrel{\beta_2 \,|\, a}{\xrightarrow{\hspace{1cm}}} s 
\]
are present in $\mathcal{T}_{E,r}$.
This situation is depicted on the left of~\cref{fig:sit1}.
\begin{figure}[ht]
\centering
\begin{tikzpicture}
[>=stealth',photon/.style={decorate,decoration={snake,post length=1mm}}]
\tikzstyle{every node}=[shape=circle,fill=none,draw=black,minimum size=18pt,inner sep=2pt]

\node(q1) at (0,2) {$t_1$};
\node(q2) at (0,0) {$t_2$};
\node(s) at (2.5,1) {$s$};

\tikzstyle{every path}=[color =black, line width = 0.5 pt]
\tikzstyle{every node}=[shape=circle,draw=none,minimum size=8pt,inner sep=1pt]

\draw [->] (q2) to [] node [below] {$\beta_2|a$}  (s);
\draw [->] (q1) to [] node [above] {$\beta_1|a$}  (s);

\begin{scope}[shift={(5,0)}]]
\tikzstyle{every node}=[shape=circle,fill=none,draw=black,minimum size=18pt,inner sep=2pt]

\node(r1) at (0,2) {$t_1$};
\node(r2) at (0,0) {$t_2$};
\node(s1) at (2.5,2) {$q_1$};
\node(s2) at (2.5,0) {$q_2$};
\node(pp) at (5,1) {$s$};

\tikzstyle{every path}=[color =black, line width = 0.5 pt]
\tikzstyle{every node}=[shape=circle,draw=none,minimum size=10pt,inner sep=2pt]

\draw [->] (r2) to [] node [below] {$\beta_2|a$}  (s2);
\draw [->] (r1) to [] node [above] {$\beta_1|a$}  (s1);
\draw [->,photon] (s2) to [] node [below] {$x_2|y$}  (pp);
\draw [->,photon] (s1) to [] node [above] {$x_1|y$}  (pp);

\end{scope}
\end{tikzpicture}
\caption{$(t_1,t_2)\in P_{s,1}$ (left) and $(t_1,t_2)\in P_{s,k+1}$ (right).}
\label{fig:sit1}
\end{figure}
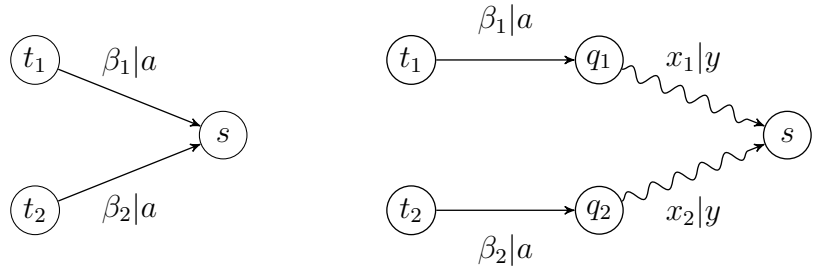 

Now, suppose that $k\ge 1$ and that the set of pairs $P_{s,k}$ has been already defined. We define $P_{s,k+1}$ as the union of $P_{s,k}$ and the set of pairs of states $(t_1,t_2)$ for which there exist a pair $(q_1,q_2)$ in $P_{s,k}$,  $\beta_1,\beta_2\in E$ (not necessarily distinct in this case) and $a\in\mathbb{N}$ such that 
the transitions
\[
    t_1\stackrel{\beta_1 \,|\, a}{\xrightarrow{\hspace{1cm}}} q_1  
    \quad \text{ and }\quad
    t_2\stackrel{\beta_2 \,|\, a}{\xrightarrow{\hspace{1cm}}} q_2 
\]
are present in $\mathcal{T}_{E,r}$.
The situation is depicted on the right of~\cref{fig:sit1}. Edges drawn with a wavy style represent a path of input $x_1$ and $x_2$ with same output $y$. Note that by induction, the last letters of $x_1$ and $x_2$ must be distinct. 

By construction and since $\mathcal{T}_{E,r}$ has finitely many states, the sequence of sets $(P_{s,k})_{k\ge 1}$ is stationary. We let $P_s$ denote the maximal so-obtained set of pairs.

Mutatis mutandis, we also define sets $(P^*_{s,k})_{k\ge 1}$ and $P^*_s$ for states $s\in \mathcal{T}^*_{E,r}$.
\end{definition}    

\begin{proposition}
\label{prop:2-walks-characterization}
Consider the assumptions of~\cref{thm:main_transducer} and let $r\in\mathbb{Q}(\delta)\cap[0,1]$. There exist two closed walks in $\mathcal{T}_{E,r}$ (resp.\ $\mathcal{T}_{E,r}^*$) starting from a same state with same output but distinct inputs if and only if there exists a state $t$ of $\mathcal{T}_{E,r}$ (resp.\ $\mathcal{T}_{E,r}^*$) such that the pair $(t,t)$ belongs to $P_t$ (resp.\ $P^*_t$).
\end{proposition}

\begin{proof}
The sufficient condition is straightforward. Now assume that $s$ is a state from which start two closed walks with the same output and distinct inputs of the form $u\beta v$ and $u'\beta'v$ such that $\beta,\beta'\in E$ are distinct letters and $u,u',v\in E^*$. Backtracking the walks, let $t$ be the first state with two incoming transitions 
\[
q_1\stackrel{\beta|a}{\xrightarrow{\hspace{1 cm}}} t \text{ and }q_2\stackrel{\beta'|a}{\xrightarrow{\hspace{1 cm}}} t.
\]
Otherwise stated, reading the input $v$ from $t$ leads to $s$. Then the state $t$ satisfies the condition $(t,t)\in P_t$.
\end{proof}

\begin{corollary}
With assumptions of~\cref{thm:main_transducer}, the property of having two closed walks with distinct input and same output is decidable in $\mathcal{O}((\# Q)^3)$ steps, where $Q$ is the set of states of $\mathcal{T}_{E,r}$ (or $\mathcal{T}^*_{E,r}$ respectively). 
\end{corollary}

\begin{proof}
    Using \cref{prop:2-walks-characterization}, we have to compute the sets $P_s$ for each $s\in Q$. Each of these $\# Q$ computations can be done in at most $(\# Q)^2$ steps. Hence the conclusion.
\end{proof}

\begin{example} 
\label{ex:no-2-walk}
Consider the Pisot number $1+\sqrt{2}$ and its square $3+2\sqrt{2}$. By a straightforward analysis, see~\cref{fig:exa2_2}, the reader may be convinced that the $2$-walk property does not hold, or apply the algorithm we have described. 

\begin{figure}[ht] 
\begin{center}
\begin{tikzpicture}[->,thick,main node/.style={circle,draw,minimum size=12pt,inner sep=2pt}]
\tikzstyle{every node}=[shape=circle,draw=red,minimum size=12pt,inner sep=2pt]

\node(a)[color=red] at (0,0) {};

\tikzstyle{every node}=[shape=circle,fill=none,draw=black,minimum size=12pt,inner sep=2pt]

\node(b1) at (3,0) {};
\node(b2) at (-3,0) {};

\tikzstyle{every path}=[color =black, line width = 0.5 pt]
\tikzstyle{every node}=[shape=rectangle,draw=none,minimum size=10pt,inner sep=2pt]
\tikzset{every loop/.style={min distance=10mm,in=30,out=-30,looseness=15}}

\draw [->] (a) to [bend left=15] node [above] {\small $1+\sqrt{2}|2$}  (b1);
\draw [->] (a) to [bend left=15] node [below] {\small $3+2\sqrt{2}|5$} (b2);

\draw [->] (b1) to [bend left=15] node [below] {\small $1+\sqrt{2}|0$}  (a);
\draw [->] (b1) to [loop right] node [right] {\small $3+2\sqrt{2}|2$} ();

\tikzset{every loop/.style={min distance=10mm,in=210,out=-210,looseness=15}}

\draw [->] (b2) to [bend left=15] node [above] {\small $1+\sqrt{2}|1$}  (a);
\draw [->] (b2) to [loop left] node [left] {\small $3+2\sqrt{2}|4$} ();

\end{tikzpicture}
\end{center}
\caption{The transducer $\mathcal{T}^*_{E,1}$ for  $E=\{1+\sqrt{2},(1+\sqrt{2})^2\}$.}
\label{fig:exa2_2}
\end{figure}
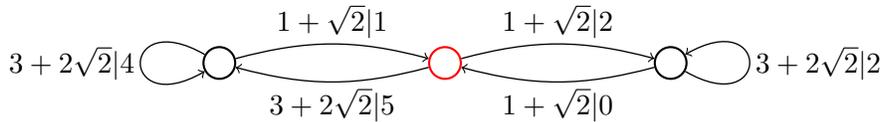

\end{example}

\begin{example} 
\label{ex:smallPisot}
We now provide an example where the $2$-walk property appears. Let $\beta$ be the smallest Pisot number, i.e., the real root of the polynomial $X^3-X-1$. Consider the set $E=\{\beta,\beta^3\}$.  We may apply \cref{cor:pisotpowers}. The transducer $\mathcal{T}^*_{E,1}$ is depicted in~\cref{fig:exa2}.

\begin{figure}[ht] 
\begin{center}
\begin{tikzpicture}[->,thick,main node/.style={circle,draw,minimum size=12pt,inner sep=2pt}]
\tikzstyle{every node}=[shape=circle,draw=red,minimum size=12pt,inner sep=2pt]

\node(a)[color=red] at (5,0) {};

\tikzstyle{every node}=[shape=circle,fill=none,draw=black,minimum size=12pt,inner sep=2pt]

\node(b1) at (0,0) {};
\node(b2) at (2.5,2.5) {};
\node(b3) at (0,5) {};
\node(b4) at (5,5) {};

\tikzstyle{every path}=[color =black, line width = 0.5 pt]
\tikzstyle{every node}=[shape=rectangle,draw=none,minimum size=10pt,inner sep=2pt]
\tikzset{every loop/.style={min distance=10mm,in=-30,out=30,looseness=10}}

\draw [->] (a) to [bend left=15] node [below left] {\small $\beta|1$}  (b2);
\draw [->,color=cyan] (a) to [bend right=15] node [above right] {\small $\beta^3|2$} (b2);

\draw [->] (b1) to [] node [left] {\small $\beta|0$}  (b3);
\draw [->,color=orange] (b1) to [] node [below] {\small $\beta^3|0$}  (a);

\draw [->,color=orange] (b2) to [] node [above left] {\small $\beta|0$}  (b1);
\draw [->,color=violet] (b2) to [] node [above left] {\small $\beta^3|0$}  (b4);

\draw [->] (b3) to [] node [above] {\small $\beta|0$}  (b4);
\draw [->] (b3) to [] node [above right] {\small $\beta^3|1$}  (b2);

\draw [->,color=violet] (b4) to [] node [right] {\small $\beta|0$}  (a);
\draw [->] (b4) to [loop] node [right] {\small $\beta^3|1$}  ();
\end{tikzpicture}
\end{center}
\caption{The transducer $\mathcal{T}^*_{E,1}$ for  $E=\{\beta,\beta^3\}$, where $\beta$ is the smallest Pisot number.}
\label{fig:exa2}
\end{figure}
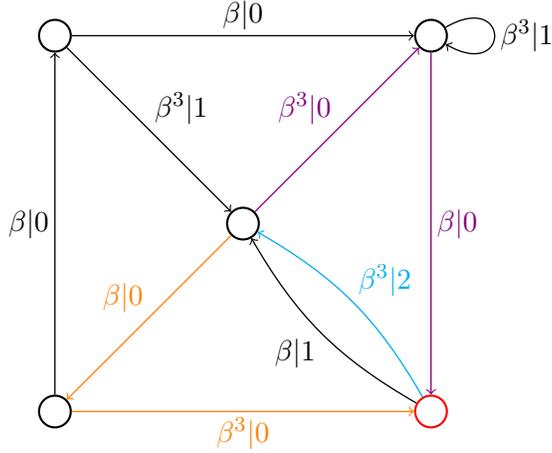
In this example, $(1,1) \in S_1$ because 
\[
    1\stackrel{\beta^3\beta\beta^3|200}{\xrightarrow{\hspace{1.5 cm}}} 1 
    \quad\text{ and }\quad
    1\stackrel{\beta^3\beta^3\beta|200}{\xrightarrow{\hspace{1.5 cm}}} 1.
\]
Let us write $u=\beta^3\beta\beta^3$ and $v=\beta^3\beta^3\beta$. These words correspond to the two highlighted closed walks in the figure. 
Therefore, any Cantor real base $\B\in\{u,v\}^\omega$ leads to $d_{\B}^{*}(1)=(200)^{\omega}$. Also note that for any such Cantor real base $\B$, we have $d_{\sigma^{3n}(\B)}^*(1)=(200)^{\omega}$ for any $n\ge 1$. 

Actually, it turns out that for $\B=uvvuvuuv\cdots$ (the Thue--Morse word over $\{u,v\}$), we have $d_{\sigma^n(\B)}^{*}(1)=w_n(200)^{\omega}$ for all $n\ge 0$, where $w_n\in E^*$ is a prefix depending on $n$. This will be made explicit in~\cref{subsec:examplesmallpisot}. 

Notice that we also have 
\[
    1\stackrel{\beta\beta\beta^3|100}{\xrightarrow{\hspace{1.5 cm}}} 1 
    \quad\text{ and }\quad
    1\stackrel{\beta\beta^3\beta|100}{\xrightarrow{\hspace{1.5 cm}}} 1.
\] 
However, this alternative choice of input words leads to aperiodic $d_{\sigma^n(\B)}^{*}(1)$ for some $n\geq 0$. This also will be made explicit in~\cref{rem:Aperiodic-112-221}. 
\end{example}

\section{Connectedness}
\label{sec:connectedness}

The transducers depicted in~\cref{fig:exa1},~\cref{fig:exa2_2} and~\cref{fig:exa2} are strongly connected. If one is interested in the complexity measure introduced in \cref{rem:complexity}, when the transducer is strongly connected, we can find Cantor real bases having full complexity (all states are visited). On the other hand, if the transducer has several strongly connected components (SCC), a Cantor base visiting only one SCC will have a smaller complexity with respect to that measure. That is a conceivable motivation, but studying the structure for its own sake is also of interest.

Recall that a Parry number $\beta$ is said to be \emph{simple} if $d_{\beta}^*(1)$ is purely periodic (or equivalently, if $d_\beta(1)$ is finite). Hence, its associated Parry automaton is strongly connected. One might thus expect that if the input Cantor real bases only contains simple Parry numbers, then the quasi-greedy transducer $\mathcal{T}^*_{E,1}$ in~\cref{thm:main_transducer} should be strongly connected. Quite surprisingly, this is not the case. In fact, it is also possible that the transducer $\mathcal{T}^*_{E,1}$ is strongly connected even though the alphabet $E$ is made up of non-simple Parry numbers.

We present some examples describing all the possible situations. Consider the Pisot numbers $\gamma_1=1+\sqrt{2}$, $ \gamma_2=2+2\sqrt{2}$, $\gamma_3=4+3\sqrt{2}$, $\gamma_4=5+4\sqrt{2}$ and $\gamma_5=1+\sqrt[3]{2}+\sqrt[3]{4}$. Let us recall that every Pisot number is a Parry number, and that every quadratic Parry number is a Pisot number, see~\cite{Bassino02}. One can computationally check the results summarized in \cref{tab:scc}.

\begin{table}[ht]
\renewcommand{\arraystretch}{1.5}
\[
\begin{tabular}{c|c|c|c}
First base & Second base & $\mathcal{T}^*_{1,E}$ is connected &  $\mathcal{T}^*_{1,E}$ is not connected \\ \hline \hline
Simple & Simple & $E=\{\gamma_1,\gamma_2\}$  & $E=\{\gamma_3,\gamma_4\}$ \\ \hline
Simple & Non simple & $E=\{\gamma_1,\gamma_1^2\}$ & $E=\{\gamma_1^2,\gamma_1^3\}$ \\ \hline 
Non simple & Non simple & $E=\{\gamma_5^2,\gamma_5^3\}$ & $E=\{\gamma_1^2,\gamma_2^2\}$ 
\end{tabular}
\]
\caption{Structure of the transducers.}\label{tab:scc}
\end{table}
\renewcommand{\arraystretch}{1}

The transducer $\mathcal{T}^*_{E,1}$ with $E=\{\gamma_1,\gamma_1^2\}$ is already depicted in~\cref{fig:exa2_2} and the remaining corresponding transducers are depicted in~\cref{fig:connectedness} except $\mathcal{T}^*_{E,1}$ with $E=\{\gamma_5^2,\gamma_5^3\}$, which has $127$ states. 

\begin{figure}[ht]
\centering
\begin{subfigure}{\textwidth}
    \centering
    \begin{tikzpicture}[->,thick,main node/.style={circle,draw,minimum size=12pt,inner sep=2pt}]
\tikzstyle{every node}=[shape=circle,draw=red,minimum size=12pt,inner sep=2pt]

\node(a0)[color=red] at (-1.5,0) {};

\tikzstyle{every node}=[shape=circle,fill=none,draw=black,minimum size=12pt,inner sep=2pt]

\node(b01) at (-4.5,0) {};
\node(b02) at (1.5,0) {};

\tikzstyle{every path}=[color =black, line width = 0.5 pt]
\tikzstyle{every node}=[shape=rectangle,draw=none,minimum size=10pt,inner sep=2pt]

\draw [->] (a0) to [bend left=40] node [above] {\small $1+\sqrt{2}|2$}  (b02);
\draw [->] (a0) to [bend left=40] node [below] {\small $2+2\sqrt{2}|4$}  (b01);

\draw [->] (b01) to [bend left=40] node [above] {\small $1+\sqrt{2}|1$}  (a0);
\draw [->] (b01) to [bend right=10] node [above] {\small $2+2\sqrt{2}|3$}  (a0);

\draw [->] (b02) to [bend left=40] node [below] {\small $1+\sqrt{2}|0$}  (a0);
\draw [->] (b02) to [bend right=10] node [below] {\small $2+2\sqrt{2}|1$}  (a0);

\end{tikzpicture}
    \caption{$\mathcal{T}^*_{E,1}$ for $E=\{1+\sqrt{2},2+2\sqrt{2}\}$}
\end{subfigure}

\begin{subfigure}{\textwidth}    
\begin{tikzpicture}[->,thick,main node/.style={circle,draw,minimum size=12pt,inner sep=2pt}]
\tikzstyle{every node}=[shape=circle,draw=red,minimum size=12pt,inner sep=2pt]

\node(a1)[color=red] at (0,0) {};

\tikzstyle{every node}=[shape=circle,fill=none,draw=black,minimum size=12pt,inner sep=2pt]

\node(b11) at (-7.5,0) {};
\node(b12) at (-5,0) {};
\node(b13) at (-2.5,0) {};
\node(b14) at (2.5,0) {};
\node(b15) at (5,0) {};

\tikzstyle{every path}=[color =black, line width = 0.5 pt]
\tikzstyle{every node}=[shape=rectangle,draw=none,minimum size=10pt,inner sep=2pt]

\draw [->] (a1) to [bend left=15] node [below] {\small $4+3\sqrt{2}|8$}  (b13);
\draw [->] (a1) to [bend left=15] node [above] {\small $5+4\sqrt{2}|10$}  (b14);

\tikzset{every loop/.style={min distance=10mm,in=60,out=120,looseness=10}}
\draw [->] (b11) to [loop] node [above] {\small $4+3\sqrt{2}|6$}  ();
\tikzset{every loop/.style={min distance=10mm,in=240,out=300,looseness=10}}
\draw [->] (b11) to [loop] node [below] {\small $5+4\sqrt{2}|8$}  ();

\draw [->] (b12) to [] node [below] {\small $4+3\sqrt{2}|4$}  (b11);
\draw [->] (b12) to [bend left=15] node [above] {\small $5+4\sqrt{2}|6$}  (b13);

\draw [->] (b13) to [bend left=15] node [above] {\small $4+3\sqrt{2}|1$}  (a1);
\draw [->] (b13) to [bend left=15] node [below] {\small $5+4\sqrt{2}|2$}  (b12);

\draw [->] (b14) to [] node [above] {\small $4+3\sqrt{2}|5$}  (b15);
\draw [->] (b14) to [bend left=15] node [below] {\small $5+4\sqrt{2}|6$}  (a1);

\tikzset{every loop/.style={min distance=10mm,in=60,out=120,looseness=10}}
\draw [->] (b15) to [loop] node [above] {\small $4+3\sqrt{2}|3$}  ();
\tikzset{every loop/.style={min distance=10mm,in=240,out=300,looseness=10}}
\draw [->] (b15) to [loop] node [below] {\small $5+4\sqrt{2}|4$}  ();
\end{tikzpicture}
\caption{$\mathcal{T}^*_{E,1}$ for $E=\{4+3\sqrt{2},5+4\sqrt{2}\}$}
\end{subfigure}

\begin{subfigure}{\textwidth}
    \centering
    \begin{tikzpicture}[->,thick,main node/.style={circle,draw,minimum size=12pt,inner sep=2pt}]
\tikzstyle{every node}=[shape=circle,draw=red,minimum size=12pt,inner sep=2pt]

\node(a)[color=red] at (0,0) {};

\tikzstyle{every node}=[shape=circle,fill=none,draw=black,minimum size=12pt,inner sep=2pt]

\node(b1) at (3,0) {};
\node(b2) at (6,0) {};
\node(b3) at (0,-3) {};
\node(b4) at (3,-3) {};
\node(b5) at (6,-3) {};

\tikzstyle{every path}=[color =black, line width = 0.5 pt]
\tikzstyle{every node}=[shape=rectangle,draw=none,minimum size=10pt,inner sep=2pt]

\draw [->] (a) to [] node [above] {\small $3+2\sqrt{2}|2$}  (b1);
\draw [->] (a) to [bend left=15] node [right] {\small $7+5\sqrt{2}|14$}  (b3);

\tikzset{every loop/.style={min distance=10mm,in=60,out=120,looseness=10}}
\draw [->] (b1) to [loop] node [above] {\small $3+2\sqrt{2}|4$}  ();
\draw [->] (b1) to [bend left=15] node [above] {\small $7+5\sqrt{2}|11$}  (b2);

\draw [->] (b2) to [bend left=15] node [below] {\small $3+2\sqrt{2}|3$}  (b1);
\draw [->] (b2) to [] node [right] {\small $7+5\sqrt{2}|9$}  (b5);

\draw [->] (b3) to [] node [below] {\small $3+2\sqrt{2}|0$}  (b4);
\draw [->] (b3) to [bend left=15] node [left] {\small $7+5\sqrt{2}|0$}  (a);

\tikzset{every loop/.style={min distance=10mm,in=240,out=300,looseness=10}}
\draw [->] (b4) to [loop] node [below] {\small $3+2\sqrt{2}|0$}  ();
\draw [->] (b4) to [] node [right] {\small $7+5\sqrt{2}|5$}  (b1);

\draw [->] (b5) to [bend left=15] node [below] {\small $3+2\sqrt{2}|1$}  (b4);
\draw [->] (b5) to [bend right=15] node [above] {\small $7+5\sqrt{2}|3$}  (b4);

\end{tikzpicture}
    \caption{$\mathcal{T}^*_{E,1}$ for $E=\{3+2\sqrt{2},7+5\sqrt{2}\}$}
\end{subfigure}

\begin{subfigure}{\textwidth}
    \centering
    \begin{tikzpicture}[->,thick,main node/.style={circle,draw,minimum size=12pt,inner sep=2pt}]
\tikzstyle{every node}=[shape=circle,draw=red,minimum size=12pt,inner sep=2pt]

\node(a)[color=red] at (0,-0.5) {};

\tikzstyle{every node}=[shape=circle,fill=none,draw=black,minimum size=12pt,inner sep=2pt]

\node(b1) at (-2.5,2) {};
\node(b2) at (2.5,2) {};

\tikzstyle{every path}=[color =black, line width = 0.5 pt]
\tikzstyle{every node}=[shape=rectangle,draw=none,minimum size=10pt,inner sep=2pt]

\draw [->] (a) to [] node [below right] {\small $3+2\sqrt{2}|5$}  (b2);
\draw [->] (a) to [] node [below left] {\small $12+5\sqrt{2}|23$}  (b1);

\tikzset{every loop/.style={min distance=10mm,in=210,out=150,looseness=10}}
\draw [->] (b1) to [bend left=15] node [above] {\small $3+2\sqrt{2}|1$}  (b2);
\draw [->] (b1) to [loop] node [left] {\small $12+8\sqrt{2}|7$}  ();

\tikzset{every loop/.style={min distance=10mm,in=-30,out=30,looseness=10}}
\draw [->] (b2) to [loop] node [right] {\small $3+2\sqrt{2}|4$}  ();
\draw [->] (b2) to [bend left=15] node [below] {\small $12+8\sqrt{2}|19$}  (b1);

\end{tikzpicture}
    \caption{$\mathcal{T}^*_{E,1}$ for $E=\{3+2\sqrt{2},12+8\sqrt{2}\}$}
\end{subfigure}
  \caption{Four of the six transducers built on Parry numbers.}
  \label{fig:connectedness}
\end{figure}
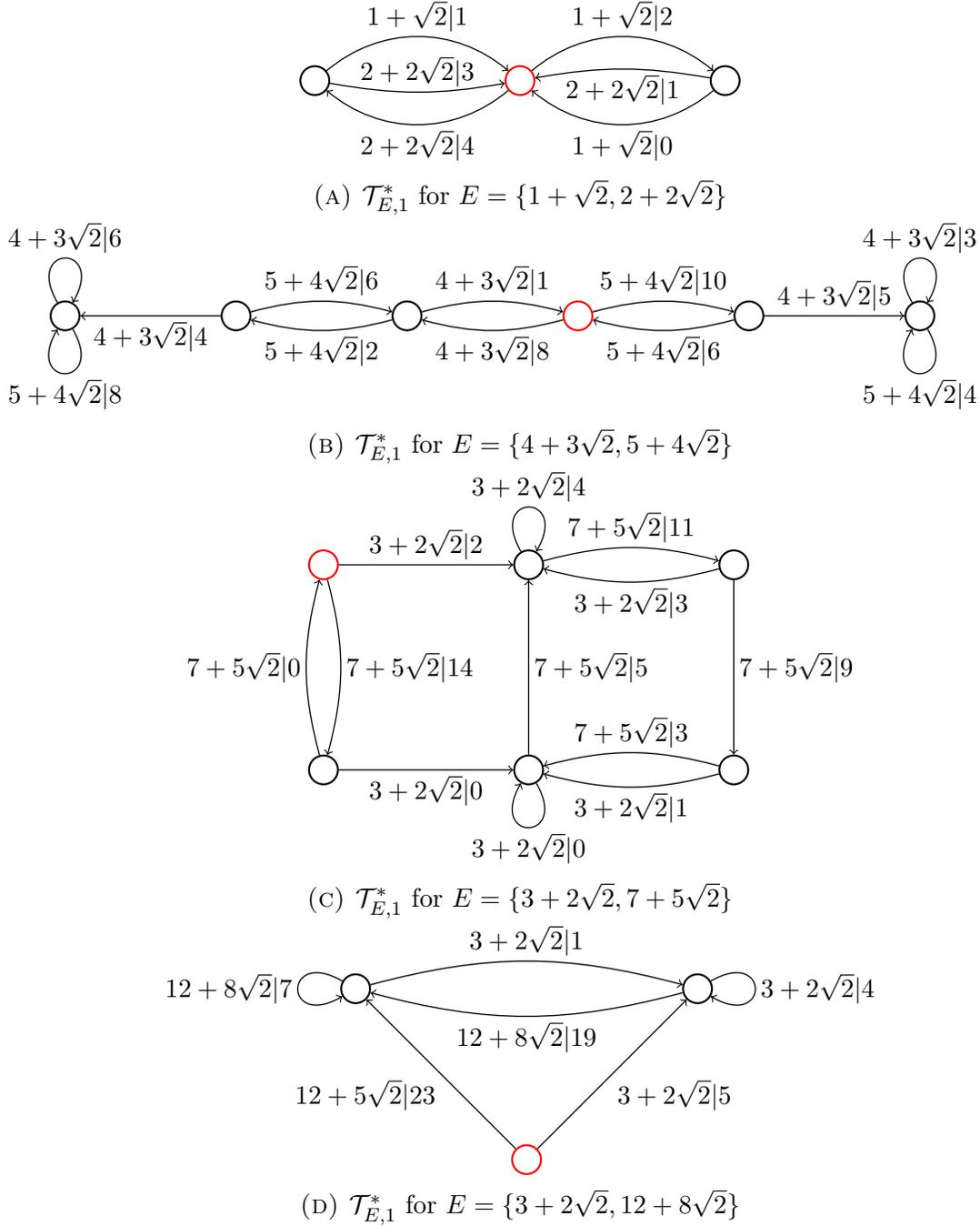

Therefore, there are no general results about the strong connectedness of the transducer by considering only if the bases are Parry numbers or not. 

\section{Decidability}\label{sec:decidability}

The main question discussed in this section is the following one. Given an infinite word (typically, an automatic sequence) can we decide whether it is a valid $\B$-expansion?

\subsection{Some background}
To be self-contained, we briefly recap some classical results about first-order logic. We refer the reader to \cite{BHMV,RigoBook} for details. 

A \emph{structure} $\mathcal{S}=\langle D, (R_i)_i\rangle$ consists of a domain $D$ and relations $R_i$ on $D$. 
A first-order formula with no free variables is a {\em sentence}.
The \emph{first-order theory} $\text{Th}(\mathcal{S})$ (also denoted by $\text{FO}(\mathcal{S})$ in other contexts) of a structure $\mathcal{S}$ is the set of true sentences of~$\mathcal{S}$. It is \emph{decidable}, if there exists an algorithm that may decide if any sentence in $\mathcal{S}$ is true or false. Otherwise, it is \emph{undecidable}. 

Let $b\ge 2$ be an integer. We let $V_b\colon \mathbb{N}\to \mathbb{N}$ be the map defined by $V_b(0)=1$ and for $n\ge 1$, $V_b(n)=i$ if $b^i$ is the greatest power of $b$ that divides $n$. We are interested in a particular structure of domain $\mathbb{N}$ and with relations given by the graphs of $+$ and $V_b$. 
A fundamental result of B\"uchi is that $\text{Th}(\langle\mathbb{N},+,V_b\rangle)$ is decidable \cite{Buchi1960}. This implies that many properties of $b$-automatic sequences are decidable. For a comprehensive presentation, we refer the reader to \cite{ShallitBook}. We state the main result that we will use and, which is in essence B\"uchi's theorem \cite[Thm.~6.4.1]{ShallitBook}

\begin{theorem}\label{thm:decidable}
    There is an algorithm that, given a formula $\phi$ with no free
variables, phrased in first-order logic, using only the universal and existential
quantifiers, addition and subtraction of variables and constants, logical operations, comparisons, and indexing into a given automatic sequence $\mathbf{x}$, will decide the truth of that formula.
\end{theorem}

This result is generalized to a larger family of numeration systems: those such that addition can be computed by a finite automaton and for which the set of greedy expansions of integers is regular. Hence the results that we will present are straightforward to adapt to this class. For instance, one can consider a sequence such as the Fibonacci word and, more generally, sequences that are automatic with respect to a Pisot numeration system \cite{BruyereHansel,RigoBook}.

Putting together \cref{thm:decidable} and \cref{cor:automatic}, if $\B$ is an automatic sequence, then we can decide many combinatorial properties of $d_{\sigma^n(\B)}^*(1)$ such as periodicity, avoiding powers, etc.

\subsection{A general (and technical) result}

An infinite sequence $\mathbf{a}=(a_n)_{n\in\mathbb{N}}$ of non-negative integers is said to be $\B$-admissible if $\mathbf{a}=d_{\B}(x)$ for some $x\in[0,1)$. The quasi-greedy expansions of $1$ can be used to determine if an infinite sequence $\mathbf{a}$ is $\B$-admissible as expressed by the following result generalizing Parry's theorem to Cantor real bases.

\begin{theorem}[\cite{CC2021}]
\label{thm:GenParry}
Let $\B$ be a Cantor real base. An infinite sequence $\mathbf{a}$ over $\mathbb{N}$ is $\B$-admissible if and only if
$\sigma^n(\mathbf{a}) <_{\mathrm{\lex}}d_{\sigma^n(\B)}^*(1)$ for all $n\in\mathbb{N}$.
 \end{theorem}

Parry's conditions from \cref{thm:GenParry} are far from being trivial to test, even in a simple situation where the Cantor base is the Thue--Morse word $\mathbf{T}=2332\cdots$ over $\{2,3\}$. While the set $\{\sigma^n(\mathbf{B}) : n\ge 0\}$ is finite whenever $\B$ is an alternate base, the set $\{\sigma^n(\mathbf{T}):n\ge 0\}$ is infinite since $\mathbf{T}$ is aperiodic. This explains why the case of alternate bases is simpler to handle.

Furthermore, the set $\{d_{\sigma^n(\B)}^*(1):n\ge 0\}$ does not seem to have any simple expression.  The following technical proposition asserts a decidable question for an automatic development in a Cantor real base. Later on, we will provide three applications of this result.

\begin{theorem} 
\label{prop:decidability}
    Let $E\subset\mathbb{R}_{>1}$ be a finite alphabet of real bases, let $\B\in E^\mathbb{N}$ be a Cantor real base, and let $b\ge 2$ be an integer base. Let $A$ be the finite alphabet of digits that can occur in a $\B$-expansion, and assume that, for all $a\in A$, we are given a first-order formula $\varphi_{a}(j,n)$ of the structure $\langle\mathbb{N}+,V_b\rangle$ that is true if and only if $[d^*_{\sigma^n{(\B)}}(1)]_j=a$, i.e., the digit at position $j$ in $d^*_{\sigma^n{(\B)}}(1)$ is $a$. 
    Then it is decidable whether any given $b$-automatic sequence over $\mathbb{N}$ is $\B$-admissible.  
\end{theorem}

\begin{proof}
Let $\mathbf{a}=a_0a_1a_2\cdots$ be a $b$-automatic sequence over $\mathbb{N}$. 
Due to~\cref{thm:GenParry}, we have to check that 
\begin{equation} 
\label{eq:lex1}
    \sigma^n(\mathbf{a})<_{\lex} d_{\sigma^n(\B)}^*(1)\quad\text{ for all } n\in\mathbb{N}.
\end{equation}
For two sequences $\mathbf{u}=u_0u_1\cdots$ and $\mathbf{v}=v_0v_1\cdots$, the lexicographic inequality $\mathbf{u}<_{\lex} \mathbf{v}$ can be expressed as \[(\exists j)\left[u_j<v_j \wedge (\forall i<j) (u_i=v_i)\right].\] 
So, the Parry conditions~\eqref{eq:lex1} become
\[ 
    (\forall n)(\exists j)\left[a_{n+j}<[d_{\sigma^n(\B)}^*(1)]_j \wedge (\forall i<j) (a_{n+i}=[d_{\sigma^n(\B)}^*(1)]_i)\right].
\] 
The inequality $a_{n+j}<[d_{\sigma^n(\B)}^*(1)]_j$
can be expressed as the first-order formula
\[
    \bigvee_{\substack{a,a'\in A\\ a<a'}} [a_{n+j}=a \wedge \varphi_{a'}(j,n)]
\]
while the equality $a_{n+i}=[d_{\sigma^n(\B)}^*(1)]_i$ can be expressed as the first-order formula
\[
    \bigvee_{a\in A} [a_{n+i}=a \wedge \varphi_a(i,n)]
\]
Therefore, we may apply \cref{thm:decidable} and conclude with the proof.
\end{proof}

\subsection{Deciding with an alternate base} 

\begin{corollary} 
\label{cor:alternate-automatic-decidability}
Let $\B$ be an alternate base $\B$ such that $d_{\sigma^i(\B)}^*(1)$ is $b$-automatic for all $i\in\{0,\ldots,p-1\}$ and let $b\ge 2$ be an integer. It is decidable whether any given $b$-automatic sequence over $\mathbb{N}$ is $\B$-admissible.
\end{corollary}

\begin{proof}
    Using \cref{prop:decidability}, it is enough to show how to define the predicate $\varphi_{a}$ as a first-order formula of $\langle\mathbb{N},+,V_b\rangle$. For each $i\in\{0,\ldots,p-1\}$ and each possible digit $a$, since $d_{\sigma^i(\B)}^*(1)$ is $b$-automatic, there exists a first-order formula $\psi_{i,a}(j)$ of the structure $\langle\mathbb{N}+,V_b\rangle$ that is true if and only of $[d_{\sigma^i(\B)}^*(1)]_j=a$. Therefore, we can define $\varphi_a(j,n)$ by
\[
    \bigvee_{i=0}^{p-1} \left( \big[(\exists q)(n=q\cdot p+i)\big]  \wedge \psi_{i,a}(j) \right).
\]
\end{proof}

\emph{Parry alternate bases} are alternate base $\B=(\beta_0,\ldots,\beta_{p-1})^\omega$ with the property that all quasi-greedy $\sigma^i(\B)$-expansions of $1$ are ultimately periodic, for $i\in\{0,\ldots,p-1\}$; see \cite{CCMP}. Such alternate bases fulfill the hypotheses of \cref{cor:alternate-automatic-decidability} since any ultimately periodic sequence is $b$-automatic for all $b\geq 2$.  
Also, note that one can easily extend this corollary to the case of an ultimately alternate base. The formula is straightforward to adapt.

\subsection{Deciding with a single transcendental base}

\cref{cor:alternate-automatic-decidability} permits us to consider some decision problems that were not considered before in the framework of $\beta$-numeration systems. Interestingly, it applies to some transcendental base~$\beta$. For example, one comes from the (shifted) characteristic sequence of powers of $2$, i.e., $\mathbf{p}_2=(p_n)_{n\ge 1}=11010001\cdots$ where $p_n=1$ if and only if $n$ is a power of $2$. This sequence is easily seen to be $2$-automatic. Moreover, all shifted sequences $\sigma^n(\mathbf{p}_2)$ are lexicographically less than $\mathbf{p}_2$ itself because  $0$-blocks are of increasing length. By a well known result of Parry~\cite{Parry1960}, there exists a unique real base $\beta>1$ such that $d_\beta^*(1)=11010001\cdots$. Let us argue that this number $\beta$ is transcendental. Let $f\colon [0,1)\to\mathbb{R}$ be the function defined by 
\[
    f(x):=\sum_{n\ge 0}x^{2^n}.
\] 
Kempner showed that for every integer $b\ge 2$, the real number $f(1/b)$ is transcendental ~\cite{Kempner1916}. Mahler then extended this result to any algebraic number in the open unit interval. 
 
\begin{theorem}\cite{Mahler1929}
\label{thm:trans}
Let $\alpha$ be an algebraic number in $(0,1)$. Then $f(\alpha)$ is transcendental. 
\end{theorem}

This result allows us to see that the real number $\beta$ we are considering in this example is transcendental. Indeed, if $1/\beta$ were algebraic then $f(1/\beta)=1$ would be transcendental. 

\subsection{Deciding with a Cantor real base}\label{subsec:examplesmallpisot}

Let us illustrate the phenomenon we want to describe using a relatively simple example. It contains the main ideas and avoids cumbersome notation.

We adopt the same notation as in~\cref{ex:smallPisot}. Recall that, in the transducer $\mathcal{T}^*_{E,1}$, the words $u=\beta^3\beta\beta^3$ and $v=\beta^3\beta^3\beta$ are both labels of closed walks starting from the state $1$ with the same output $200$. Let us consider the Cantor base~$\B$ obtained as the Thue--Morse sequence over $\{u,v\}$, i.e.,
\begin{align}
\B=uvvuvu\cdots=\beta^3\beta\beta^3 \beta^3 \beta^3 \beta  \beta^3 \beta^3\beta\beta^3 \beta \beta^3 \beta^3 \beta^3 \beta \beta^3 \beta \beta^3\cdots.\label{eq:tmmodif}
\end{align}
We have already seen that $d_{\sigma^{3n}(\B)}^*(1)=(200)^{\omega}$ for all $n\ge 0$. The computations tend to show that for all $n\ge 0$, there exists a prefix $w_n$ such that $d_{\sigma^n(\B)}^*(1)=w_n(200)^{\omega}$, and that these prefixes have bounded length. Hence, we should have finitely many distinct such prefixes $w_n$. They are listed in~\cref{tab:prefixes}, by order of appearance.
\begin{table}[ht]
\renewcommand{\arraystretch}{1.2} 
\[
\begin{tabular}{l|l}
    $w_n$ & $n$ \\ \hline \hline
    $\varepsilon$ & $0,3,6,9,12,\ldots$ \\
    $10110$ & $1, 10, 19, 31, 37,\ldots$  \\
    $2010$ & $2, 11, 20, 32, 38,\ldots$  \\
    $20020010110$ & $4, 40, 76, 124, 148,\ldots$ \\
    $1010$ & $5, 16, 23, 28, 41,\ldots$  \\
    $20010110$ & $7, 34, 43, 58, 79,\ldots$  \\
    $1002010$ & $8, 35, 44, 59, 80,\ldots$  \\
    $20010102010$ & $13, 25, 49, 67, 85,\ldots$ \\
    $1002002010$ & $14, 26, 50, 68, 86,\ldots$  \\
    $10102010$ & $16, 28, 52, 70, 88,\ldots$ \\
    $2002010$ & $17, 29, 53, 71, 89,\ldots$ \\
    $20020010102010$ & $22, 64, 94, 112, 166,\ldots$ \\
\end{tabular}
\]
\caption{Possible prefixes $w_n$ of $d_{\sigma^n(\B)}^{*}(1)=w_n(200)^{\omega}$.}
\label{tab:prefixes}
\renewcommand{\arraystretch}{1}
\end{table}

Notice that the longest prefix in \cref{tab:prefixes} has length $14$. The following lemma proves that all prefixes $w_n$ indeed belong to \cref{tab:prefixes}. 

\begin{lemma}
\label{lem:finitepref}
    Let $\B$ be the Cantor real base \eqref{eq:tmmodif} over the alphabet of real numbers $\{\beta,\beta^3\}$ given in \cref{ex:smallPisot}. For all $n\ge 0$, there exists $w$ in \cref{tab:prefixes} such that $d_{\sigma^n(\B)}^*(1)=w(200)^{\omega}$. 
\end{lemma}

\begin{proof}
Inspecting \cref{fig:exa2}, we see that $d_{\sigma^{3n}(\B)}^{*}(1)=(200)^{\omega}$ for all $n\ge 0$, i.e., the preperiod is $\varepsilon$. Indeed, we have two highlighted closed walks starting from the initial state with input $u$ or $v$ and the same output $200$. For the remaining cases, we have a closer look at the set of factors of length $5$ of the Thue--Morse word, which give rise to factors of length $15$ in $\B$. This set of factors, by order of appearance within $\mathbf{B}$, is  
\begin{align*} 
F=&\{f_1=uvvuv,f_2=vvuvu,f_3=vuvuu,f_4=uvuuv,\\ 
    & f_5=vuuvv,f_6=uuvvu,f_7=uvvuu,f_8=vvuuv, \\
    & f_9=vuuvu,f_{10}=uuvuv,f_{11}=uvuvv,f_{12}=vuvvu\}.
\end{align*} 
Since our computations in \cref{tab:prefixes} have shown that the longest preperiod should be of length $14$, this suggests that it should be enough to check these factors.

Precisely, for $i\in\{1,2\}$, $\sigma^{3n+i}(\B)$ starts with a prefix of the form $\sigma^i(f)$ of length $15-i$ for some $f\in F$, where the notation $\sigma^i(f)$ means that we drop the first $i$ letters of the finite word $f$. We can check that, in the transducer of \cref{fig:exa2}, these 24 possible prefixes label a closed walk starting from $1$ and output one of the $12$ prefixes of \cref{tab:prefixes} padded with suffixes in $(200)^*$. For example, we have $\sigma(f_1)=\beta\beta^3 vvuv$ and the prefix $\beta\beta^3 v=\beta\beta^3\beta^3\beta^3\beta$ already labels a closed walk starting from $1$, which has $10110$ as output. The output corresponding to $\sigma(f_1)$ is $10110200200200$, which is $10110$ padded three times with $200$ to obtain a word of length $14=15-1$. The output prefixes corresponding to all possible $\sigma^i(f_m)$ are given in \cref{tab:allprefshift}. We see that all are indeed of the form $w(200)^*$ with $w$ belonging to \cref{tab:prefixes}.
\begin{table}[ht]
\renewcommand{\arraystretch}{1.2} 
\[
\begin{tabular}{c|c||c|c}
    Input prefix & Output prefix & Input prefix & Output prefix \\ \hline \hline
    $\sigma(f_1)$ & $10110200200200$ & $\sigma^2(f_1)$ & $2010200200200$ \\ 
    $\sigma(f_2)$ & $20020010110200$ & $\sigma^2(f_2)$ & $1010200200200$ \\ 
    $\sigma(f_3)$ & $20010110200200$ & $\sigma^2(f_3)$ & $1002010200200$ \\
    $\sigma(f_4)$ & $10110200200200$ & $\sigma^2(f_4)$ & $2010200200200$ \\ 
    $\sigma(f_5)$ & $20010102010200$ & $\sigma^2(f_5)$ & $1002002010200$ \\   
    $\sigma(f_6)$ & $10102010200200$ & $\sigma^2(f_6)$ & $2002010200200$ \\   
    $\sigma(f_7)$ & $10110200200200$ & $\sigma^2(f_7)$ & $2010200200200$ \\   
    $\sigma(f_8)$ & $20020010102010$ & $\sigma^2(f_8)$ & $1010200200200$ \\   
    $\sigma(f_9)$ & $20010102010200$ & $\sigma^2(f_9)$ & $1002002010200$ \\   
    $\sigma(f_{10})$ & $10102010200200$ & $\sigma^2(f_{10})$ & $2002010200200$ \\   
    $\sigma(f_{11})$ & $10110200200200$ & $\sigma^2(f_{11})$ & $2010200200200$ \\    
    $\sigma(f_{12})$ & $20010110200200$ & $\sigma^2(f_{12})$ & $1002010200200$ 
\end{tabular}
\]
\caption{Output prefixes associated with words in $F$ and a shift of $1$ or $2$.}\label{tab:allprefshift}
\renewcommand{\arraystretch}{1}
\end{table}

After reading this prefix, we are back to the state $1$ of the transducer and the infinite input suffix that is left to read belongs to $\{u,v\}^\omega$. Hence, the corresponding output for this suffix is $(200)^\omega$. This concludes with the proof.
\end{proof}

This example provides a non-trivial example of decidability with an aperiodic Cantor real base. 

\begin{proposition}
Let $\B$ be the Cantor real base \eqref{eq:tmmodif} over the alphabet of real numbers $\{\beta,\beta^3\}$ given in \cref{ex:smallPisot}.  
It is decidable whether any given $2$-automatic sequence over $\mathbb{N}$ is $\B$-admissible.
\end{proposition}

\begin{proof}
    To apply~\cref{prop:decidability}, we have to show that the property $\varphi_a(j,n)$ is definable in $\langle\mathbb{N},+,V_2\rangle$ for each $a\in\{0,1,2\}$. The formula is made of several parts, which may depend on $a$. 
    First, we proceed to Euclidean divisions of $j$ and $n$ by $3$:
     \[
        \bigvee_{r,r'\in\{0,1,2\}} \Big([(\exists q)(\exists q')(n=q\cdot 3+r\wedge j=q'\cdot 3+r')]\wedge 
        \gamma_{a,1}(q,r,j)\wedge\gamma_{a,2}(r,j,r')
        \Big)
    \] 
    where the formula $\gamma_{a,1}(q,r,j)$ is intended to handle the preperiodic part of $d_{\sigma^n(\B)}^*(1)$ while the formula $\gamma_{a,2}(r,j,r')$ will deal with its periodic part.
        
    Since the Thue--Morse word is $2$-automatic, for each $m\in\{1,\ldots,12\}$, there is a first-order formula $\psi_m(q)$ of $\langle\mathbb{N},+,V_2\rangle$ which is true if and only if the factor of length $5$ of the Thue--Morse word over $\{u,v\}$ occurring at position~$q$ is $f_m$ (using the notation of the proof of \cref{lem:finitepref}). For each $r\in\{1,2\}$ and $m\in\{1,\ldots,12\}$, we let $w_{r,m}$ be the output prefix of length $15-r$ corresponding to $\sigma^r(f_m)$ given in \cref{tab:allprefshift}.  
    
    If $r=1$ and $j\le 13$, or if $r=2$ and $j\le 12$, then we have to check that $a$ occurs at position $j$ in $w_{r,m}$. This finite amount of information can be coded in a formula depending on $a$: we set 
    \[
        \gamma_{a,1}(q,r,j)\equiv [(r=1\wedge j\le 13)\vee (r=2\wedge j\le 12)] \implies
            \left( \bigvee_{m=1}^{12} \big(\psi_m(q)\wedge [w_{r,m}]_j=a\big)\right).
    \]

    Otherwise, we are in the periodic part $(200)^\omega$ and we have to deal with the shifts. This is the purpose of the formula $\gamma_{a,2}$. Depending on $a$, we set
    \begin{align*}
        \gamma_{0,2}(r,j,r') 
            & \equiv [r=0\implies (r'=1\vee r'=2)] \\
            & \qquad \wedge [(r=1 \wedge j\ge 14) \implies (r'=0\vee r'=1)] \\
            & \qquad \wedge [(r=2 \wedge j\ge 13) \implies (r'=0\vee r'=2)] \\
        \gamma_{1,2}(r,j,r') 
            & \equiv r\ne 0 \wedge \neg(r=1 \wedge j\ge 14) \wedge \neg(r=2 \wedge j\ge 13), \\
        \gamma_{2,2}(r,j,r') 
            & \equiv [r=0\implies r'=0]\wedge [(r=1 \wedge j\ge 14) \implies r'=2] \\
            & \qquad \wedge [(r=2 \wedge j\ge 13) \implies r'=1].
    \end{align*}
    
    We have thus described the different parts of a first-order formula defining $\varphi_a$. In particular, note that the function~$V_2$ is only needed to handle factors of the Thue--Morse word.
\end{proof}

The idea behind the previous proof can be synthesized in the following more general statement. With these new notation, in the latter proof, we took advantage of the $b$-automaticity of the Cantor real base $\B$ combined with the structural properties of the transducer $\mathcal{T}_{E,1}^*$ to be able to define the finitely many formulas $\psi_w(n)$ handling the possible preperiods. Note that there was only one possible periodic part $x=200$ in this situation.

\begin{proposition}
\label{pro:finprop}
    Let $b\ge 2$ be an integer base and let $\B$ be a Cantor real base over a finite alphabet of real bases.  
    Suppose that $\B$ has the following finiteness property: there exists a finite language $L$ such that for all $n\ge 0$, there exist $w_n,x_n\in L$ such that $d_{\sigma^n(\B)}^*(1)=w_n(x_n)^\omega$. 
    Furthermore, suppose 
    for each $w\in L$ (resp.\ $x\in L$), there exists a formula $\psi_w(n)$ (resp.\ $\chi_x(n)$) of the structure $\langle \mathbb{N},+,V_b\rangle$ that is true if and only if $w_n=w$ (resp.\ $x_n=x$).
    Then it is decidable whether any given $b$-automatic sequence is $\B$-admissible.
\end{proposition}

\begin{proof}
    Again, we aim at using \cref{thm:decidable}. Let $A$ be the finite alphabet of digits that may occur in $B$-expansions of real numbers in $[0,1]$ and let $a\in A$. The predicate $\varphi_a(j,n)$ can be defined within the structure $\langle \mathbb{N},+,V_b\rangle$ as follows:
    \[
        \bigvee_{w,x\in L} [\psi_w(n)\wedge \chi_x(n)\wedge \gamma_{a,w}(j)\wedge \delta_{a,w,x}(j)]
    \]
    where 
    \[
        \gamma_{a,w}(j)\equiv j<|w| \implies [w]_j=a
    \]
    handles the preperiod 
    and 
    \[
        \delta_{a,w,x}(j)\equiv \bigvee_{r=0}^{|w|-1} [(\exists q)(j=|w|+q\cdot|x|+r)\wedge [x]_r=a]
    \]
    handles the period.
\end{proof}



\begin{remark} 
\label{rem:Aperiodic-112-221}
    In Lemma~\ref{lem:finitepref}, we proved that the finiteness property mentioned in Proposition~\ref{pro:finprop} holds for the Cantor real base $\B$ from \cref{ex:smallPisot} given in~\eqref{eq:tmmodif}.
    Let us emphasize here that our proof heavily depends on the input words $u$ and $v$ used to build the Cantor real base $\B$. Consider $u'=\beta\beta\beta^3$ and $v'=\beta\beta^3\beta$ instead. These two words still label closed walks starting from the initial state $1$ with a common output label, which is $100$ in this case. Even though $d_{\sigma^{3n}(\B)}^*(1)=(100)^\omega$ for all $n\ge 0$ when $\B$ is the Cantor base $\B=u'v'v'u'v'u'u'v'\cdots$ (obtained as the Thue--Morse word over $\{u',v'\}$), we find ourselves with an aperiodic $\sigma(\B)$-expansion
    \begin{align*}
        d_{\sigma(\B)}^*(1)=100200100010200010100200100010100200010200\cdots.
    \end{align*} 
    Indeed, the sequence $\B$ is $2$-automatic since it is the image under a uniform morphism of a $2$-automatic sequence, see~\cref{cor:automatic}. In turn, $\sigma(\B)$ is also $2$-automatic as a shifted $2$-automatic sequence. By~\Cref{cor:automatic}, we get that the word $d_{\sigma(\B)}^*(1)$ is $2$-automatic. This is a routine computation to determine that the automaton generating this word has $51$ states, for instance with the {\tt Mathematica} package {\tt IntegerSequences} \cite{IntSeq}. We can thus automatically check that this word is aperiodic.  Evaluating the following command in the automatic theorem prover \texttt{Walnut}
\begin{verbatim}
eval periodic "Ep,n (p>=1) & (Ai (i>=n) => D[i]=D[i+p])":
\end{verbatim} where \texttt{D} is the automaton that generates $d_{\sigma(\B)}^{*}(1)$, returns {\tt False}, see~\cite{ShallitBook} for more details. Therefore, we cannot hope for a result such as \cref{lem:finitepref} for the Cantor real base $\B$ built from $u'$ and $v'$, and thus the finiteness hypothesis in \cref{pro:finprop} is not satisfied. 
\end{remark}

\section{Acknowledgements}
We thank Kevin Hare for helping us with algebraic properties of Pisot numbers. 
 Émilie Charlier is supported by the FNRS grant J.0034.22.
 Pierre Popoli is supported by ULiège's Special Funds for Research, IPD-STEMA Program. 
 Michel Rigo is supported by the FNRS Research grant T.196.23 (PDR). 

\bibliographystyle{alpha}
\bibliography{biblio.bib}

\end{document}